\documentclass[reqno]{amsart}

\usepackage{xfrac}

\usepackage{hyperref,todonotes,comment,multicol}
\usepackage[color,matrix,arrow]{xy}
\usepackage{graphics,graphicx,amssymb}
\usepackage{pagecolor,lipsum}
\usepackage{latexsym}
\usepackage{mathrsfs}
\xyoption{curve}
\usepackage{hyperref}
\usepackage{xcolor}
\hypersetup{ 
colorlinks,
citecolor=blue,
filecolor=blue,
linkcolor=blue,
urlcolor=blue
}

\newtheorem{lemma}{Lemma}[section]
\newtheorem{proposition}[lemma]{Proposition}
\newtheorem{theorem}[lemma]{Theorem}
\newtheorem{corollary}[lemma]{Corollary}

\newtheorem*{theoremA}{Theorem}

\theoremstyle{definition}
\newtheorem{example}[lemma]{Example}
\newtheorem{definition}[lemma]{Definition}
\newtheorem{remark}[lemma]{Remark}

\newcommand{\mfk}[1]{\mathfrak{#1}}
\newcommand{\mbb}[1]{\mathbb{#1}}
\newcommand{\mcl}[1]{\mathcal{#1}}

\newcommand{\opn}[1]{\operatorname{#1}}
\newcommand{\ot}{\otimes}

\DeclareMathOperator{\Hom}{Hom}

\DeclareMathOperator{\Ind}{Ind}
\DeclareMathOperator{\Fun}{Fun}

\newcommand{\Frac}{\mathcal{F}\!rac}

\renewcommand{\1}{\mathbf{1}}

\newcommand{\nocontentsline}[3]{}
\newcommand{\tocless}[2]{\bgroup\let\addcontentsline=\nocontentsline#1{#2}\egroup}

\definecolor{page_color}{HTML}{000000}
\definecolor{text_color}{HTML}{F0EAD6}

\makeindex

\title[Cocompletions for non-abelian tensor categories]{Cocompletions for non-abelian vertex tensor categories}
\date{\today}

\author{Robert McRae}
 \address{Yau Mathematical Sciences Center, Tsinghua University, Beijing 100084, China}
  \email{rhmcrae@tsinghua.edu.cn}
\author{Cris Negron}
\address{Department of Mathematics, University of Southern California, Los Angeles, CA 90007}
\email{cnegron@usc.edu}

\begin{document}

\begin{abstract}
It was recently shown by Huang that the category of $C_1$-cofinite modules for any vertex operator algebra $V$ admits a natural braided monoidal structure.
Here, we show that this structure extends uniquely to a vertex algebraically natural braided monoidal structure on the completion of the category of $C_1$-cofinite $V$-modules under filtered colimits, within the ambient category of all generalized $V$-modules. 
A critical point here is that we do not assume the category of $C_1$-cofinite $V$-modules is abelian or that $C_1$-cofinite modules are compact in the cocompletion, since these properties are not known to hold in general. Our results have many applications in the representation theory of vertex operator algebra extensions, since many vertex operator algebras can be realized as objects in the filtered colimit completion of the category of $C_1$-cofinite modules for a vertex operator subalgebra. Generalizing from the specific vertex algebraic setting, we also establish existence and uniqueness for extensions of monoidal structures along a dense inclusion $\mcl{C}_0 \to \mcl{C}$ from an abstract, essentially small monoidal category into a well-structured cocomplete target.
\end{abstract}

\maketitle

\tableofcontents

\section{Introduction}

Given an essentially small, finitely cocomplete category $\mcl{C}_0$, a natural question is how to embed $\mcl{C}_0$ into a cocomplete category $\mcl{C}$. When $\mcl{C}_0$ also has a monoidal or braided monoidal structure, one would also like $\mcl{C}$ to be monoidal such that the embedding $\mcl{C}_0\hookrightarrow\mcl{C}$ is a monoidal functor. In fact, it is always possible to do this taking $\mcl{C}$ to be the Ind-completion $\opn{Ind}(\mcl{C}_0)$, as described in \cite[Chapter 6]{kashiwaraschapira06}, for example. (In the case that $\mcl{C}_0$ is monoidal, we review how $\opn{Ind}(\mcl{C}_0)$ inherits a monoidal structure from $\mcl{C}_0$ in Section \ref{sec:Ind-comp-monoidal} below.)

However, the Ind-completion $\opn{Ind}(\mcl{C}_0)$ is rather abstract and may not be a suitable cocompletion of $\mcl{C}_0$ for some applications. For example, if objects of $\mcl{C}_0$ are concrete representations of some algebraic structure, then one may want objects of the cocompletion $\mcl{C}$ to be representations of the same algebraic structure. But objects of $\opn{Ind}(\mcl{C}_0)$ as it is normally defined are not vector spaces equipped with the action of an algebraic structure; rather they are functors from essentially small filtered categories to $\mcl{C}_0$. Thus we seek different, possibly more concrete, cocomplete monoidal categories into which $\mcl{C}_0$ embeds.

In this paper, we consider a general situation where a monoidal category $\mcl{C}_0$ is a full subcategory of an ambient cocomplete abelian category $\mcl{B}$. (We do not assume $\mcl{B}$ has a monoidal structure.) Our main result is that under reasonable conditions, the full subcategory $\mcl{C}\subseteq\mcl{B}$ consisting of filtered colimits of objects of $\mcl{C}_0$ is cocomplete and admits a natural monoidal structure such that the embedding $\mcl{C}_0\hookrightarrow\mcl{C}$ is a monoidal functor. The analogous result also holds when $\mcl{C}_0$ is braided. Our main application is the case that $\mcl{C}_0$ is a category of grading-restricted modules for a vertex operator algebra $V$, in which case the ambient category $\mcl{B}$ is the category of all generalized $V$-modules.

Similar results on braided monoidal structure for cocompletions of vertex algebraic braided monoidal categories were proved in \cite{CMY-completions}, but the conditions imposed on $\mcl{C}_0$ there were excessively restrictive; in particular, the arguments in \cite{CMY-completions} assumed that $\mcl{C}_0$ was abelian. Here, we remove this restriction, proving our results for finitely cocomplete but not necessarily abelian monoidal categories $\mcl{C}_0$.

\subsection{Results for vertex operator algebras}
Let $\mcl{C}_0$ be a category of grading-restricted generalized modules for a vertex operator algebra $V$ that admits the vertex algebraic braided monoidal structure described in \cite{HLZ8}. Then $\mcl{C}_0$ lives in the ambient cocomplete category $\mcl{B}$ of generalized $V$-modules, which are $V$-modules that decompose into generalized eigenspaces for the Virasoro $L_0$-operator, with no restriction on the dimensions of the generalized eigenspaces. The category $\mcl{B}$ is in general too big to be the desired cocompletion of $\mcl{C}_0$; indeed, it is not known in general whether $\mcl{B}$ admits the braided monoidal structure of \cite{HLZ8}. Instead, we consider the minimal possible cocompletion of $\mcl{C}_0$ within $\mcl{B}$, namely, we take $\mcl{C}$ to be the full subcategory of $\mcl{B}$ consisting of all filtered colimits of objects in $\mcl{C}_0$. More concretely, objects of $\mcl{C}$ are precisely those generalized $V$-modules that are the unions of submodules which are objects of $\mcl{C}_0$.

It was shown in \cite{CMY-completions} that in this setting, and under further conditions, the subcategory $\mcl{C}$ indeed admits the vertex algebraic braided monoidal structure of \cite{HLZ8}, and that the embedding $\mcl{C}_0\hookrightarrow\mcl{C}$ is braided monoidal. However, recently, one of the further conditions in \cite{CMY-completions} has turned out to be a problem: it was assumed in \cite{CMY-completions} that $\mcl{C}_0$ is an abelian category (in particular, closed under taking submodules). The problem here is that Huang has recently shown in \cite{Hu-C1} that the category of so-called $C_1$-cofinite modules for any vertex operator algebra $V$ admits the braided monoidal structure of \cite{HLZ8}, but this category is not known to be abelian in general. Thus we cannot apply the results of \cite{CMY-completions} to the category $\mcl{C}_0$ of $C_1$-cofinite $V$-modules in general, even though this is a braided monoidal category.

In this paper, we solve this problem by removing the assumption that $\mcl{C}_0$ is abelian. Specifically, our main result for vertex operator algebras is:
\begin{theoremA}[\ref{thm:vertex-mon-for-C0}]
Let $V$ be a vertex operator algebra, let $\mcl{C}_0$ be an essentially small full subcategory of the category $\mcl{B}$ of generalized $V$-modules, and let $\mcl{C}$ be the full subcategory of $\mcl{B}$ whose objects are the unions of submodules which are objects of $\mcl{C}_0$. Assume also that:
\begin{enumerate}

\item $\mcl{C}_0$ is closed under finite coproducts and arbitrary quotients.

\item $\mcl{C}_0$ admits the vertex algebraic braided monoidal structure of \cite{HLZ8}.

\item For any intertwining operator $\mcl{Y}$ of type $\binom{M_3}{M_1\,M_2}$ where $M_1$, $M_2\in\opn{obj}(\mcl{C}_0)$ and $M_3\in\opn{obj}(\mcl{C})$, the image of $\mcl{Y}$ is contained in a $\mcl{C}_0$-submodule of $M_3$.

\item All objects of $\mcl{C}_0$ are finitely generated as generalized $V$-modules.

\end{enumerate}
Then $\mcl{C}$ is additive and cocomplete and admits the vertex algebraic braided monoidal structure of \cite{HLZ8}. Furthermore, the inclusion $\mcl{C}_0\hookrightarrow\mcl{C}$ is braided monoidal and the tensor product $\mcl{C}\times\mcl{C}\rightarrow\mcl{C}$ is cocontinuous in each factor.
\end{theoremA}

The category of $C_1$-cofinite modules for any vertex operator algebra satisfies the conditions of this theorem, so we also get:

\begin{theoremA}[\ref{thm:braid-mon-for-C1}]
Let $\mcl{C}_0$ be the category of $C_1$-cofinite modules for any vertex operator algebra $V$, and let $\mcl{C}$ be the full subcategory of generalized $V$-modules whose objects are the unions of submodules which are objects of $\mcl{C}_0$. Then $\mcl{C}$ is additive and cocomplete, and $\mcl{C}$ admits the vertex algebraic braided monoidal structure of \cite{HLZ8}. Furthermore, the inclusion $\mcl{C}_0\hookrightarrow\mcl{C}$ is braided monoidal and the tensor product $\mcl{C}\times\mcl{C}\rightarrow\mcl{C}$ is cocontinuous in each factor.
\end{theoremA}

\subsection{Applications to vertex operator algebra extensions}
We expect Theorems \ref{thm:vertex-mon-for-C0} and \ref{thm:braid-mon-for-C1} to have many applications in the study of extensions of vertex operator algebras. It was shown in \cite{HKL} that if $\mcl{C}_0$ is a braided monoidal category of modules for a vertex operator algebra $V$, and if $A$ is a vertex operator algebra that contains $V$ as a vertex operator subalgebra, such that $A$ viewed as a $V$-module is an object of $\mcl{C}_0$, then $A$ has the structure of a commutative algebra in the braided monoidal category $\mcl{C}_0$. This allows one to use methods and results on commutative algebras in braided monoidal categories to study representations of $A$ based on the representation theory of $V$, and vice versa \cite{CKM}. However, this works only if $A$ is actually an object of a braided monoidal category of $V$-modules, and if this braided monoidal structure is the vertex algebraic one of \cite{HLZ8}.

Thanks to \cite{Hu-C1}, one can start with the braided monoidal category $\mcl{C}_0$ of $C_1$-cofinite modules for a vertex operator algebra $V$ and look for extensions $A\supseteq V$. But in many cases, one finds extensions $A$ that are not objects of $\mcl{C}_0$, but rather are infinite direct sums (or more general unions) of $C_1$-cofinite $V$-modules. This occurs, for example, with the vertex algebra of chiral differential operators on an algebraic group \cite{MSV}, which is an extension of the tensor product of two affine vertex operator algebras. For such examples, Theorem \ref{thm:braid-mon-for-C1} now shows that indeed the extended algebra $A$ is an object of a suitable vertex algebraic braided monoidal category of $V$-modules, namely the cocompletion $\mcl{C}$ of $\mcl{C}_0$ inside the category of generalized $V$-modules, and therefore the extension theory of \cite{HKL, CKM} can be used to relate the representation theories of $A$ and $V$. In the example of the vertex algebra of chiral differential operators on a simple algebraic group at a rational level, the base category $\mcl{C}_0$ of $C_1$-cofinite modules for the tensor product of two affine vertex operator algebras is not expected to be abelian in general, so Theorem \ref{thm:braid-mon-for-C1} is truly needed here; the earlier results of \cite{CMY-completions} are not sufficient.

Another application of Theorem \ref{thm:braid-mon-for-C1} occurs in the recent work \cite{CDN}, where fusion rules of $C_1$-cofinite modules for affine $W$-algebras of $ADE$ type at irrational levels are obtained through analyzing a certain extension of the tensor product of the principal affine $W$-algebra with a corresponding affine vertex operator algebra. In this example, it turns out that the category of $C_1$-cofinite modules for the principal $W$-algebra is abelian, but this is not known \textit{a priori}, and thus Theorem \ref{thm:braid-mon-for-C1} is required rather than the earlier results proved in \cite{CMY-completions}. This example was in fact our original motivation for proving Theorems \ref{thm:vertex-mon-for-C0} and \ref{thm:braid-mon-for-C1}.

\subsection{Abstract results}\label{subsec:abs-set}

Theorem \ref{thm:vertex-mon-for-C0} is, after some specific analysis, a special case of a completely formal result which we believe to be of independent interest. The generic situation is as follows: We consider an essentially small, full subcategory $\mcl{C}_0$ in an abelian category $\mcl{B}$.  We assume that $\mcl{B}$ admits all colimits and that $\mcl{C}_0$ is stable under sums and quotients in $\mcl{B}$, and hence finite colimits as well.  We take $\mcl{C}\subseteq \mcl{B}$ to be the full subcategory spanned by those objects in $\mcl{B}$ which are obtained as filtered colimits of objects in $\mcl{C}_0$.  In this situation, an object lives in $\mcl{C}$ precisely when it's expressible as the union of its subobjects which lie in $\mcl{C}_0$, $M=\bigcup_\alpha M_\alpha$.
\par

We consider the case where $\mcl{C}_0$ admits a (braided) monoidal structure which commutes with finite colimits in each factor, and which preserves surjections.

\begin{theoremA}[\ref{thm:abs_v1}]
Consider $\mcl{C}_0$ and $\mcl{C}$ as above, with $\mcl{C}_0$ (braided) monoidal. Suppose that for any filtered colimit $\varinjlim_\alpha M_\alpha=M$ which occurs in $\mcl{C}_0$, and any $N$ in $\mcl{C}_0$, the natural maps
\[
\varinjlim_\alpha (M_\alpha\ot N)\to M\ot N\ \ \text{and}\ \ \varinjlim_\alpha (N\ot M_\alpha)\to N\ot M
\]
are isomorphisms.  Suppose also that all objects in $\mcl{C}_0$ are finitely generated (Definition \ref{def:fg}).  Then there is a unique (braided) monoidal structure on $\mcl{C}$ for which the following hold:
\begin{enumerate}
\item[(a)] The inclusion $i:\mcl{C}_0\to \mcl{C}$ is (braided) monoidal with trivial structure maps $i(M)\ot i(N)\overset{=}\to i(M\ot N)$.
\item[(b)] The product functor $\ot:\mcl{C}\times\mcl{C}\to \mcl{C}$ commutes with small colimits in each factor.
\end{enumerate}
\end{theoremA}

The proof is somewhat delicate, but fairly straightforward in its trajectory.  First, according to general principles, the monoidal structure on $\mcl{C}_0$ lifts to a monoidal structure on its Ind-completion $\opn{Ind}(\mcl{C}_0)$.  We then show that, under the above assumptions, the unique map $\opn{Ind}(\mcl{C}_0)\to \mcl{C}$ induced by the inclusion $\mcl{C}_0\to \mcl{C}$ realizes $\mcl{C}$ as a localization
\[
\opn{Ind}(\mcl{C}_0)[W^{-1}]\overset{\sim}\to \mcl{C}
\]
relative to a certain class of morphisms in $\opn{Ind}(\mcl{C}_0)$ (Corollary \ref{cor:470}).  We argue that the class $W$ is stable under the product on $\opn{Ind}(\mcl{C}_0)$ (Proposition \ref{prop:loc_tensor}), so that the above localization $\opn{Ind}(\mcl{C}_0)[W^{-1}]\cong \mcl{C}$ inherits a monoidal structure from that of $\opn{Ind}(\mcl{C}_0)$.

\begin{remark}
In the specific case of representation categories for vertex operator algebras, we provide an alternate proof of Theorem \ref{thm:braid-mon-for-C1} which uses ideas from the proof of Theorem \ref{thm:abs_v1} in conjunction with arguments from \cite{CMY-completions}, and which avoids most categorical nonsense. The interested reader can see Section \ref{sec:another-proof} for details.
\end{remark}

\subsection{Outline}

We now outline the remaining contents of this paper. In Section \ref{sect:cats}, we review definitions and results from category theory that we will need, including Ind-completions, finitely generated objects in an abelian category, and isofibrations.

In Section \ref{sec:Ind-comp-monoidal}, we review how to give the Ind-completion $\opn{Ind}(\mcl{C}_0)$ of a (braided) monoidal, essentially small, and finitely cocomplete category $\mcl{C}_0$ the structure of a (braided) monoidal category such that the inclusion $\mcl{C}_0\hookrightarrow \opn{Ind}(\mcl{C}_0)$ is (braided) monoidal and such that the product on $\opn{Ind}(\mcl{C}_0)$ is cocontinuous in each factor.

In Section \ref{sect:big-objects}, we consider a cocomplete abelian category $\mcl{B}$ with two full subcategories $\mcl{C}_0\subseteq\mcl{C}\subseteq\mcl{B}$ where $\mcl{C}_0$ is essentially small and stable under finite colimits in $\mcl{B}$, and $\mcl{C}$ consists of all objects in $\mcl{B}$ which are obtained as small colimits of objects in $\mcl{C}_0$. Under suitable conditions, we show that $\mcl{C}$ is equivalent to a localization of $\mathrm{Ind}(\mcl{C}_0)$. Specifically, we show that $\mcl{C}$ is the localization of $\mathrm{Ind}(\mcl{C}_0)$ by the class of $\mcl{C}$-relative equivalences, which are the morphisms in $\opn{Ind}(\mcl{C}_0)$ that get sent to isomorphisms in $\mcl{C}$ under the natural functor $\opn{Ind}(\mcl{C}_0)\rightarrow\mcl{C}$ induced by the inclusion $\mcl{C}_0\hookrightarrow \mcl{C}$.

In Section \ref{sec:monoidality}, we remain in the setting of Section \ref{sect:big-objects}. We show that under suitable further conditions, the class of $\mcl{C}$-relative equivalences is stable under the product on $\mathrm{Ind}(\mcl{C}_0)$. Using this, we show that the (braided) monoidal structure on $\mathrm{Ind}(\mcl{C}_0)$ induces a (braided) monoidal structure on its localization $\mcl{C}$, thus completing the proof of Theorem \ref{thm:abs_v1}.

In Section \ref{sec:vertex-alg}, we first recall definitions from vertex operator algebra representation theory that we will need, and then we check the conditions of Theorem \ref{thm:abs_v1} for suitable categories of generalized modules for vertex operator algebras, thus proving Theorems \ref{thm:vertex-mon-for-C0} and \ref{thm:braid-mon-for-C1}.

In Section \ref{sec:another-proof}, we give an alternative more direct proof of Theorem \ref{thm:braid-mon-for-C1} that does not explicitly use Theorem \ref{thm:abs_v1}. Instead, we observe that to prove Theorem \ref{thm:braid-mon-for-C1}, it is sufficient to follow the previous arguments from \cite{CMY-completions} and replace any that use the assumption that $\mcl{C}_0$ is abelian. In fact, there is only one key proposition in \cite{CMY-completions} that essentially uses this assumption, and we give it a new proof using elements from the proofs of Propositions \ref{prop:348} and \ref{prop:loc_tensor} here.

\medskip

\noindent\textbf{Acknowledgments.} This work was made possible by Institut Pascal at Université Paris-Saclay with the support of the program ``Investissements d’avenir'' ANR-11-IDEX-0003-01, since the work was begun while we attended the conference \textit{CFT: Algebraic, Topological and Probabilistic Approaches to Correlators in Conformal Field Theory} at the Institut Pascal. We thank the organizers of the conference for inviting us to attend, and we thank Thomas Creutzig for introducing this problem to us at the conference. Cris Negron was supported by NSF CAREER Grant DMS-2239698 and Simons Collaboration Grant 999367.

\section{Categorical background}
\label{sect:cats}

We recall basic information on cocomplete categories, cocompletions, and isofibrations.  Isofibrations are employed throughout the text in order to control the ambiguities which arise when transferring structures along universal functors.

\subsection{Standard construction for Ind-completions}
\label{sect:stand_ind}

Recall that a category $K$ is called filtered if every pair of objects in $K$ admits a cone, and every pair of maps $f,f':\alpha\to \beta$ in $K$ admits a third map $g:\beta\to \kappa$ for which $g\circ f =g\circ f'$.  A filtered colimit in a category $\mcl{B}$ is a colimit indexed by a filtered category.
\par

For essentially small $\mcl{C}_0$, the standard construction of the Ind-completion $\opn{Ind}(\mcl{C}_0)$ is as the category of diagrams $K\to \mcl{C}_0$ from essentially small filtered categories $K$.  For objects $M_\ast:K\to \mcl{C}_0$ and $N_\ast:L\to \mcl{C}_0$ in $\opn{Ind}(\mcl{C}_0)$ the morphisms are given by
\[
\Hom_{\opn{Ind}(\mcl{C}_0)}(M_\ast,N_\ast)=\varprojlim_\alpha\varinjlim_\beta\Hom_{\mcl{C}_0}(M_\alpha,N_\beta).
\]
We have the fully faithful inclusion $\mcl{C}_0\to \opn{Ind}(\mcl{C}_0)$, $M\mapsto (M:\ast\to \mcl{C}_0)$.

We note that $\opn{Ind}(\mcl{C}_0)$ admits all filtered colimits. Furthermore, when $\mcl{C}_0$ admits finite colimits, $\opn{Ind}(\mcl{C}_0)$ admits all small colimits \cite[Proposition 6.1.18]{kashiwaraschapira06}.

\begin{definition}
Given a category $\mcl{C}$ which admits small filtered colimits, $M\in\mathrm{obj}(\mcl{C})$ is called compact if for each small filtered colimit $N=\varinjlim_\beta N_\beta$ the natural map
\[
\varinjlim_\beta \Hom_{\mcl{C}}(M, N_\beta)\to \Hom_{\mcl{C}}(M,N)
\]
is an isomorphism.
\end{definition}

When $\mcl{C}_0$ is idempotent split, the inclusion $\mcl{C}_0\to \opn{Ind}(\mcl{C}_0)$ identifies $\mcl{C}_0$ with the subcategory of compact objects in $\opn{Ind}(\mcl{C}_0)$.  More generally the compacts in $\opn{Ind}(\mcl{C}_0)$ are identified with the idempotent splitting of $\mcl{C}_0$ in $\opn{Ind}(\mcl{C}_0)$ \cite[Proposition 5.40]{kelly82}.

\subsection{Cocomplete categories, etc.}

\begin{definition}
We call a category $\mcl{C}$ cocomplete (resp.\ finitely cocomplete) if it admits all small colimits (resp.\ finite colimits).  A functor between cocomplete categories $F:\mcl{C}\to \mcl{D}$ is called cocontinuous if it commutes with small colimits.  A functor between finitely cocomplete categories is called finitely cocontinuous if it commutes with finite colimits.  In the respective cases we let
\[
\Fun_{cc}(\mcl{C},\mcl{D})\ \ \text{and}\ \ \Fun_{fcc}(\mcl{C},\mcl{D})
\]
denote the full subcategories of cocontinuous and finitely cocontinuous functors in $\Fun(\mcl{C},\mcl{D})$, respectively.
\end{definition}

When $\mcl{C}_0$ is finitely cocomplete the inclusion $\mcl{C}_0\to\opn{Ind}(\mcl{C}_0)$ is finitely cocontinuous \cite[Corollary 6.1.6]{kashiwaraschapira06}.  So, given any cocomplete category $\mcl{D}$, restriction provides a natural functor
\[
\Fun_{cc}(\opn{Ind}(\mcl{C}_0),\mcl{D})\to \Fun_{fcc}(\mcl{C}_0,\mcl{D}).
\]
In this cocomplete setting--which is the setting we are interested in--the universal property of the Ind-completion appears as follows.

\begin{theorem}\label{thm:ind_univ}
Let $\mcl{C}_0$ be essentially small and finitely cocomplete. For any cocomplete category $\mcl{D}$, restriction provides an equivalence of categories
\[
\Fun_{cc}(\opn{Ind}(\mcl{C}_0),\mcl{D})\overset{\sim}\to \Fun_{fcc}(\mcl{C}_0,\mcl{D}).
\]
\end{theorem}

\begin{proof}
Since $\opn{Ind}(\mcl{C}_0)$ is generated by $\mcl{C}_0$ under filtered colimits, it is clear that the restriction functor is fully faithful.  So we need only establish essential surjectivity.
\par

We have the Yoneda embedding $\mcl{C}_0\to \Fun(\mcl{C}_0^{\opn{op}},\opn{Set})$ which identifies, via evaluating colimits, $\opn{Ind}(\mcl{C}_0)$ with the full subcategory in $\Fun(\mcl{C}_0^{\opn{op}},\opn{Set})$ spanned by those functors which are obtained as filtered colimits of the representables. After making this replacement, essential surjectivity follows by \cite[Proposition 5.39]{kelly82} and \cite[Proposition 1.45 (ii)]{adamekrosicky}.
\end{proof}

Generally, we speak of \emph{an} Ind-completion of $\mcl{C}_0$ as the choice of a cocomplete category $\mcl{C}'$ and a functor $F:\mcl{C}_0\to \mcl{C}'$ for which restriction along $F$ produces an equivalence as above.  In particular, if one has some preferred construction of such an Ind-completion $\mcl{C}_0\to \mcl{C}'$ we might \emph{replace} the standard construction from above, and write
\[
\opn{Ind}(\mcl{C}_0)=\mcl{C}'
\]
instead.  Indeed, we already came across such a change of models for the Ind-completion in the proof of Theorem \ref{thm:ind_univ}. We discuss another situation which is relevant for us below.

\subsection{Relative realizations of $\opn{Ind}(\mcl{C}_0)$}
\label{sect:rel_ind}

Suppose $\mcl{C}_0$ is a full subcategory in a cocomplete category $\mcl{C}$.  Suppose also that all objects in $\mcl{C}$ are generated by objects in $\mcl{C}_0$ under filtered colimits, and that $\mcl{C}_0$ is essentially small.  A $\mcl{C}_0$-atlas for an object $M$ in $\mcl{C}$ is a colimit diagram
\[
M_A:K\star \{\infty\}\to \mcl{C}
\]
in which $K$ is essentially small and filtered, $M_A|_{K}$ has image in $\mcl{C}_0$, and $M_A(\infty)=M$. (Here $K\star\{\infty\}$ is the category $K$ with a terminal object $\infty$ adjoined.) We can define a category $\mcl{C}_{\opn{atlas}}$ whose objects are objects $M$ in $\mcl{C}$ with a choice of $\mcl{C}_0$-atlas $M_A$, and whose morphisms are given by
\[
\Hom_{\mcl{C}_{\opn{atlas}}}(M_A,N_B) = \varprojlim_\alpha\varinjlim_\beta\Hom_{\mcl{C}}(M_\alpha,N_\beta).
\]
We have the apparent forgetful map to the standard Ind-completion
\[
\mcl{C}_{\opn{atlas}} \to \opn{Ind}(\mcl{C}_0)
\]
which is an equivalence.  (Full faithfulness is by construction and essential surjectivity follows by cocompleteness of $\mcl{C}$.)  We also have the two inclusions from $\mcl{C}_0$ which fit into a diagram
\[
\xymatrix{
&\mcl{C}_0\ar[dl]\ar[dr] & \\
\mcl{C}_{\opn{atlas}}\ar[rr]^\sim & & \opn{Ind}(\mcl{C}_0).
}
\]
Supposing $\mcl{C}_0$ is finitely cocomplete, and given any cocomplete category $\mcl{D}$, it follows that we have a diagram
\[
\xymatrix{
\Fun_{cc}(\opn{Ind}(\mcl{C}_0),\mcl{D})\ar[rr]^{\sim}\ar[dr]_{\sim} & & \Fun_{cc}(\mcl{C}_{\opn{atlas}},\mcl{D})\ar[dl]\\
 & \Fun_{fcc}(\mcl{C}_0,\mcl{D}) & 
}
\]
and so conclude that the restriction functor
\[
\Fun_{cc}(\mcl{C}_{\opn{atlas}},\mcl{D})\to \Fun_{fcc}(\mcl{C}_0,\mcl{D})
\]
is an equivalence.  This is to say, $\mcl{C}_{\opn{atlas}}$ is an Ind-completion of $\mcl{C}_0$.
\par

Throughout the text we are interested in such a pairing of a finitely cocomplete category $\mcl{C}_0$ with an embedding $\mcl{C}_0\to \mcl{C}$ into a cocomplete target.  In such a setting we choose the relative model for the Ind-completion,
\[
\opn{Ind}(\mcl{C}_0)=\mcl{C}_{\opn{atlas}}.
\]
\par

The only important points here are that $\mcl{C}_{\opn{atlas}}$ has the correct universal property and that the structure map $\mcl{C}_0\to \mcl{C}_{\opn{atlas}}$ is injective (by which we mean injective on objects and morphisms).  We note that in this setting the unique cocontinuous extension of the inclusion $\mcl{C}_0\to \mcl{C}$ to $\mcl{C}_{\opn{atlas}}$ is simply the functor
\[
\pi:\opn{Ind}(\mcl{C}_0)=\mcl{C}_{\opn{atlas}}\to \mcl{C}
\]
which forgets atlases on objects, $M_A\mapsto M$.

\begin{definition}\label{def:reduced}
A colimit diagram $M_A:K\star\{\infty\}\to \mcl{C}$ is called reduced if, for each structure map $M_\alpha\to M$ and factorization $M_\alpha\to \bar{M}_\alpha\to M$ through an epimorphism $M_\alpha\to \bar{M}_\alpha$ in $\mcl{C}$, the map $M_\alpha\to \bar{M}_\alpha$ is an isomorphism.  A $\mcl{C}_0$-atlas $M_A$ for an object $M$ in $\mcl{C}$ is called reduced if $M_A$ is reduced as a colimit diagram.
\end{definition}

In all cases of interest, reducedness simply means that all of the transition maps $M_\alpha\to M_\beta$ are injective, in either the set theoretic or abelian sense of the term.

\begin{remark}
The subscript $A$ in the notation $M_A$ is just a marker which indicates that we have chosen a particular ``form" of the object $M$, or more directly some filtered system over $M$ which realizes it as a colimit of objects in $\mcl{C}_0$. As above, given any filtered colimit diagram $M_A:K\star\{\infty\}\to \mcl{C}$ and object $\alpha$ in $K$, we take $M_{\alpha}=M_A(\alpha)$.
\end{remark}

\subsection{Finitely generated objects}

\begin{definition}\label{def:fg}
Let $\mcl{B}$ be an abelian category which admits small filtered colimits.  We say an object $M$ in $\mcl{B}$ is finitely generated if, for any expression as a reduced filtered colimit $\varinjlim_\alpha M_\alpha = M$, there is an index $\gamma$ at which the structure map $M_\gamma\to M$ is an isomorphism.
\end{definition}

In any concrete setting, an object $M$ is finitely generated if and only if in any expression of $M$ as a filtered union of subobjects
\[
M=\cup_{\alpha} M_\alpha
\]
we have $M=M_\gamma$ at some index $\gamma$.

\begin{example}
For $R$ a commutative ring, an object $M$ in $R$-mod is finitely generated, in the categorical sense, if and only if $M$ is finitely generated as an $R$-module.
\end{example}

\subsection{Fiber products for categories}

The pullback
\[
\xymatrix{
\mcl{C}_0\times_{\mcl{D}}\mcl{C}_1\ar[r]\ar[d] & \mcl{C}_1\ar[d]^{F_1}\\
\mcl{C}_0\ar[r]_{F_0} & \mcl{D}
}
\]
of a pair of functors with a common target is the category whose objects and morphisms are given by the expected fiber products,
\[
\opn{obj}(\mcl{C}_0\times_{\mcl{D}}\mcl{C}_1)=\opn{obj}(\mcl{C}_0)\times_{\opn{obj}(\mcl{D})}\opn{obj}(\mcl{C}_1)
\]
and
\[
\Hom_{\mcl{C}_0}(M_0,N_0)\times_{\Hom_{\mcl{D}}(X,Y)}\Hom_{\mcl{C}_1}(M_1,N_1).
\]
Here $F_i(M_i)=X$, $F_i(N_i)=Y$, and composition is defined in the unique manner so that each projection $\mcl{C}_0\times_{\mcl{D}}\mcl{C}_1\to \mcl{C}_i$ is a functor.

\begin{remark}
If one considers categories as simplicial sets, via the nerve functor for example, then the fiber product of a pair of categories is simply given by the fiber product of simplicial sets.
\end{remark}

In general the pullback construction is poorly behaved at the level of the $2$-category of categories.  It is, however, a reasonable thing to consider whenever one of the $F_i$ is an \emph{isofibration}.

\subsection{Isofibrations}

\begin{definition}[{\cite[\href{https://kerodon.net/tag/01EN}{01EN}]{kerodon}}]
A functor $F:\mcl{C}\to \mcl{D}$ is called an isofibration if, for every object $N$ in $\mcl{C}$ and isomorphism $\bar{\alpha}:X\to F(N)$ in $\mcl{D}$, there is an isomorphism $\alpha:M\to N$ in $\mcl{C}$ with $F(\alpha)=\bar{\alpha}$.
\end{definition}

Clearly, by taking inverses, we see that $F$ is an isofibration if and only if one can lift isomorphisms $\bar{\alpha}:F(M)\to Y$ along $F$.  So this notion is symmetric.  The following is trivial to check.

\begin{lemma}\label{lem:pb_isofib}
If $F:\mcl{C}\to \mcl{D}$ is an isofibration, then for any functor $\mcl{E}\to \mcl{D}$ the projection $\mcl{C}\times_{\mcl{D}}\mcl{E}\to \mcl{E}$ is also an isofibration.
\end{lemma}

For us there are two important properties of isofibrations.  First, isofibrations appear via restriction along faithful functors.

\begin{proposition}[{\cite[\href{https://kerodon.net/tag/01F3}{01F3}]{kerodon}}]\label{prop:inj_isofib}
Let $i:\mcl{C}_0\to \mcl{C}$ be a faithful functor which is also injective on objects.  Then for any category $\mcl{D}$ the restriction functor
\[
i^\ast:\Fun(\mcl{C},\mcl{D})\to \Fun(\mcl{C}_0,\mcl{D})
\]
is an isofibration.
\end{proposition}

In the cocomplete setting we can restrict our attention to cocontinuous functors.

\begin{corollary}\label{cor:inj_isofib}
Let $i:\mcl{C}_0\to \mcl{C}$ be a finitely cocontinuous functor from a finitely cocomplete category $\mcl{C}_0$ to a cocomplete category $\mcl{C}$.  Suppose also that $i$ is faithful and injective on objects.  Then for any cocomplete category $\mcl{D}$ the restriction functor
\[
i^\ast:\Fun_{cc}(\mcl{C},\mcl{D})\to \Fun_{fcc}(\mcl{C}_0,\mcl{D})
\]
is an isofibration.
\end{corollary}

\begin{proof}
We have the commuting diagram
\[
\xymatrix{
\Fun_{cc}(\mcl{C},\mcl{D})\ar[rr]\ar[drr] & & \Fun_{fcc}(\mcl{C}_0,\mcl{D})\times_{\Fun(\mcl{C}_0,\mcl{D})}\Fun(\mcl{C},\mcl{D})\ar[d]\\
 & & \Fun_{fcc}(\mcl{C}_0,\mcl{D})
}
\]
in which the top horizontal map is fully faithful and has image stable under isomorphism.  So the result follows by Lemma \ref{lem:pb_isofib} and Proposition \ref{prop:inj_isofib}.
\end{proof}

The second important point about isofibrations is the following.

\begin{proposition}\label{prop:iso_equiv}
Consider a pullback diagram
\[
\xymatrix{
\mcl{C}_0\times_{\mcl{D}}\mcl{C}_1\ar[r]^(.55){p_1}\ar[d]_{p_0} & \mcl{C}_1\ar[d]^{F_1}\\
\mcl{C}_0\ar[r]_{F_0} & \mcl{D}
}
\]
in which one of the $F_i$ is an isofibration.  If $F_0$ is an equivalence then $p_1$ is an equivalence, and if $F_1$ is an equivalence then $p_0$ is an equivalence.
\end{proposition}

\begin{proof}
Suppose $F_0$ is an equivalence, for example.  Then $F_0$ is fully faithful and one sees directly from the description for the fiber product that the projection
\[
p_1:\mcl{C}_0\times_{\mcl{D}}\mcl{C}_1\to \mcl{C}_1
\]
is fully faithful.  So one only has to deal with essential surjectivity.
\par

Supposing $F_1$ is an isofibration, then for any $M$ in $\mcl{C}_1$ with image $X$ in $\mcl{D}$, we have an isomorphism $X'\to X$ between $X$ and an object in the image of $\mcl{C}_0$.  This follows from essential surjectivity of $F_0$.  We can lift this isomorphism to an isomorphism $M'\to M$ in $\mcl{C}_1$, and by assumption $X'$ lifts to some $M'_0$ in $\mcl{C}_0$.  We then have the object $(M'_0,M')$ in the fiber product which is sent to $M'$ by $p_1$.  Hence we have essential surjectivity.
\par

On the other hand, if $F_0$ is an isofibration, then essential surjectivity implies that $F_0$ is actually surjective on objects.  In this case the projection $p_1$ is surjective onto the objects in $\mcl{C}_1$, and again we have essential surjectivity.
\end{proof}

\section{Extending monoidality to the Ind-completion}\label{sec:Ind-comp-monoidal}

Below we consider any model for the Ind-completion for which the structure map $\mcl{C}_0\to \opn{Ind}(\mcl{C}_0)$ is injective.  So, for example, we can consider the standard construction from Section \ref{sect:stand_ind} or the relative construction from Section \ref{sect:rel_ind}.

\begin{proposition}\label{prop:ind_product}
Suppose $\mcl{C}_0$ is (braided) monoidal, essentially small, and finitely cocomplete. Suppose also that the product on $\mcl{C}_0$ commutes with finite colimits in each factor.  Then there exists a unique (braided) monoidal structure on $\opn{Ind}(\mcl{C}_0)$ for which the following hold:
\begin{enumerate}
\item[(a)] The inclusion $\opn{incl}:\mcl{C}_0\to \opn{Ind}(\mcl{C}_0)$ is (braided) monoidal with trivial tensor compatibility, $\opn{incl}(M)\ot\opn{incl}(N)\overset{=}\to \opn{incl}(M\ot N)$.\vspace{1mm}
\item[(b)] The product on $\opn{Ind}(\mcl{C}_0)$ commutes with small colimits in each factor.
\end{enumerate}
\end{proposition}

This result is well-known and oft employed, at least in the abstract.  We highlight below how one lifts the product on $\mcl{C}_0$ to $\opn{Ind}(\mcl{C}_0)$ in a \emph{strict} manner, as required by point (a) above.

\subsection{$n$-fold cocontinuous functors}

For cocomplete categories $\mcl{D}_i$ and $\mcl{E}$ we take
\[
\Fun_{n\text{-}cc}(\prod_{i=1}^n\mcl{D}_i,\mcl{E})\ \subseteq\ \Fun(\prod_{i=1}^n\mcl{D}_i,\mcl{E})
\]
the full subcategory of functors which commute with small colimits in each factor.  To clarify, we mean that at each index $j$ and each tuple of objects
\[
\{X_1,\dots,X_{j-1},X_{j+1},\dots,X_n\}
\]
the functor
\[
F(X_1,\dots,X_{j-1},-,X_{j+1},\dots,X_n):\mcl{D}_j\to \mcl{E}
\]
is cocontinuous.  We adopt a similar notation in the finitely cocomplete setting.

\begin{lemma}
Let $\mcl{C}_i$ be essentially small and finitely cocomplete, and $\mcl{C}_i\to \opn{Ind}(\mcl{C}_i)$ be Ind-completions for which the structure maps are injective.  For each $n>0$, and each cocomplete category $\mcl{D}$, restriction provides an equivalence
\[
\Fun_{n\text{-}cc}(\prod_{i=1}^n\opn{Ind}(\mcl{C}_i),\mcl{D})\overset{\sim}\to \Fun_{n\text{-}fcc}(\prod_{i=1}^n\mcl{C}_i,\mcl{D})
\]
which is furthermore an isofibration.
\end{lemma}

\begin{proof}
Restriction is an isofibration by Proposition \ref{prop:inj_isofib} and the fact that the given classes of functors are stable under natural isomorphism.  We note that any functor category $\Fun(\mcl{C},\mcl{D})$ with cocomplete target is itself cocomplete, with colimits calculated pointwise \cite[Theorem V.3.1]{maclane98}.  The full subcategory of (finitely) cocontinuous functors is a full subcategory in $\Fun(\mcl{C},\mcl{D})$ which is stable under computation of colimits, via the generic commutation relation
\[
\opn{colim}_{\alpha\in A}\opn{colim}_{\beta\in B}Y_{\alpha\beta}=\opn{colim}_{(\alpha,\beta)\in A\times B}Y_{\alpha\beta}=\opn{colim}_{\beta\in B}\opn{colim}_{\alpha\in A}Y_{\alpha\beta}.
\]
In particular, $\Fun_{cc}(\opn{Ind}(\mcl{C}_i),\mcl{D})$ and $\Fun_{fcc}(\mcl{C}_i,\mcl{D})$ are cocomplete.  By induction on $n$, and the adjunction
\[
\Fun_{n\text{-}cc}(\prod_i\opn{Ind}(\mcl{C}_i),\mcl{D}) \cong\Fun_{cc}\big(\opn{Ind}(\mcl{C}_n),\Fun_{(n-1)\text{-}cc}(\prod_{j<n}\opn{Ind}(\mcl{C}_j),\mcl{D})\big),
\]
we see that the $n$-fold cocontinuous functor categories are cocomplete as well.  Cocompleteness of $\Fun_{n\text{-}fcc}(\prod_i \mcl{C}_i,\mcl{D})$ is observed similarly.
\par

The result at general $n$ now follows from the result at $n=1$ and the successive diagrams
\[
\xymatrix{
\Fun_{n\text{-}cc}(\prod_i\opn{Ind}(\mcl{C}_i),\mcl{D})\ar[d]_{\opn{res}} \ar[r]^(.325){\sim}& \Fun_{cc}\big(\opn{Ind}(\mcl{C}_n),\Fun_{(n-1)\text{-}cc}(\prod_{j<n}\opn{Ind}(\mcl{C}_j),\mcl{D})\big)\ar[d]^\sim\\
\Fun_{n\text{-}fcc}(\prod_i\mcl{C}_i,\mcl{D}) \ar[r]^(.375){\sim}&  \Fun_{fcc}\big(\mcl{C}_n,\Fun_{(n-1)\text{-}fcc}(\prod_{j<n}\mcl{C}_j,\mcl{D})\big).
}
\]
\end{proof}

\subsection{Extending monoidal structures}

\begin{proof}[Proof of Proposition \ref{prop:ind_product}]
From the product composed with inclusion
\[
\opn{incl}\circ\ot:\ast\to \Fun_{2\text{-}fcc}(\mcl{C}_0\times\mcl{C}_0,\opn{Ind}(\mcl{C}_0))
\]
we pull back to obtain an equivalence
\[
\{\opn{incl}\circ\ot\}\times_{\Fun_{2\text{-}fcc}(\mcl{C}_0^2,\opn{Ind}(\mcl{C}_0))}\Fun_{2\text{-}cc}(\opn{Ind}(\mcl{C}_0)^2,\opn{Ind}(\mcl{C}_0))\overset{\sim}\to \ast,
\]
by Proposition \ref{prop:iso_equiv}.  Choosing any object in this fiber product gives a bicocontinuous product $\tilde{\ot}$ on $\opn{Ind}(\mcl{C}_0)$ which is unique up to unique isomorphism, and which fits into a strictly commuting diagram
\[
\xymatrix{
\opn{Ind}(\mcl{C}_0)\times\opn{Ind}(\mcl{C}_0)\ar[r]^(.6){\tilde{\ot}}& \opn{Ind}(\mcl{C}_0)\\
\mcl{C}_0\times\mcl{C}_0\ar[u]\ar[r]_(.6){\ot} & \mcl{C}_0\ar[u].
}
\]
\par

The unit object $\1_{\mcl{C}_0}\to \mcl{C}_0$ gives an object in $\1_{\mcl{C}}=\opn{ind}(\1_{\mcl{C}_0})$ in $\opn{Ind}(\mcl{C}_0)$, and we uniquely lift the left and right unit transformations
\[
\opn{incl}\circ\opn{unit}_{\mcl{C}_0}:\Delta^1\to \Fun_{fcc}(\mcl{C}_0,\mcl{C}_0)\to \Fun_{fcc}\big(\mcl{C}_0,\opn{Ind}(\mcl{C}_0)\big)
\] 
via the equivalences
\[
\{\opn{incl}\circ\opn{unit}_{\mcl{C}_0}\}\times_{\Fun_{fcc}(\mcl{C}_0,\opn{Ind}(\mcl{C}_0))}\Fun_{cc}\big(\opn{Ind}(\mcl{C}_0),\opn{Ind}(\mcl{C}_0)\big)\overset{\sim}\to \Delta^1.
\]
One similarly lifts the associator uniquely to $\opn{Ind}(\mcl{C}_0)$ via the equivalence
\[
\{\opn{incl}\circ\opn{associator}_{\mcl{C}_0}\}\times_{\Fun_{3\text{-}fcc}(\mcl{C}_0^3,\opn{Ind}(\mcl{C}_0))}\Fun_{3\text{-}cc}\big(\opn{Ind}(\mcl{C}_0)^3,\opn{Ind}(\mcl{C}_0)\big)\overset{\sim}\to \Delta^1,
\]
and also lifts the braiding when it is present.  All compatibilities between the units, associator, etc.\ hold for $\opn{Ind}(\mcl{C}_0)$ via cocontinuity of the product, the fact that $\mcl{C}_0$ generates $\opn{Ind}(\mcl{C}_0)$ under colimits, and the fact that all compatibilities are satisfied on the generating subcategory $\mcl{C}_0\hookrightarrow\opn{Ind}(\mcl{C}_0)$.
\end{proof}

\section{Big objects via localization}\label{sect:big-objects}

Consider a dense embedding $\mcl{C}_0\to \mcl{C}$ from a finitely cocomplete, essentially small category to a cocomplete target. We show that, under some restrictive but naturally occurring constraints, the target category $\mcl{C}$ is recovered as a localization of $\opn{Ind}(\mcl{C}_0)$ along a distinguished class of morphisms.  In particular, the embedding $\mcl{C}_0\to \mcl{C}$ has a universal property amongst functors $\mcl{C}_0\to \mcl{D}$ to a cocomplete category $\mcl{D}$.

\subsection{The abstract setting}\label{sect:setting}
We consider the following generic situation: Let $\mcl{B}$ be a cocomplete abelian category with two full subcategories
\[
\mcl{C}_0\subseteq \mcl{C}\subseteq \mcl{B},
\]
where $\mcl{C}_0$ is essentially small and stable under finite colimits in $\mcl{B}$, and $\mcl{C}$ consists of all objects $M$ in $\mcl{B}$ which are obtained from $\mcl{C}_0$ via small colimits.  We suppose additionally that:
\begin{itemize}
\item[(C1)] All objects in $\mcl{C}_0$ are finitely generated (when considered as objects in the ambient category $\mcl{B}$).
\item[(C2)] $\mcl{C}_0$ is stable under taking quotients in $\mcl{B}$.
\item[(C3)] Filtered colimits in $\mcl{B}$ are exact.
\end{itemize}
In this setting an object $M$ is in $\mcl{C}$ if and only if $M$ is expressed as a reduced filtered colimit of objects in $\mcl{C}_0$.  Indeed, in any concrete situation, $M$ is in $\mcl{C}$ if and only if $M$ is the union $\bigcup_\alpha M_\alpha$ of its subobjects which lie in $\mcl{C}_0$.

\begin{definition}\label{def:M-Can}
For an object $M$ in $\mcl{C}$ as above, the canonical $\mcl{C}_0$-atlas for $M$ is the colimit diagram
\begin{equation}\label{eq:330}
M_{\opn{Can}}:K_M\star \{\infty\}\to \mcl{C},
\end{equation}
where $K_M$ is the full subcategory in the overcategory $\mcl{C}_{/M}$ consisting of injections $f_\alpha:M_\alpha\to M$ from an object in $\mcl{C}_0$.
\end{definition}

The map \eqref{eq:330} here is the apparent one.  Namely, $M_{\opn{Can}}|_{K_M}$ is just the restriction of the forgetful functor $\mcl{C}_{/M}\to \mcl{C}$, the cone point $\infty$ is sent to $M$, and the unique map $\{f_\alpha:M_\alpha\to M\}\to \infty$ at each $f_{\alpha}$ in $K_M$ is sent to $f_\alpha:M_\alpha\to M$.

{\color{violet}
\begin{remark}
Of course, in order for $M_{\opn{Can}}$ to be a $\mcl{C}_0$-atlas we need $K_M$ to be essentially small. In a concrete setting one can observe essential smallness of $K_M$ directly. In general, one can argue essential smallness by embedding $\mcl{C}$ into the functor category $\Fun(\mcl{C}_0,\mbb{Z}\text{-mod})$ then observing that the analog of $K_M$ in $\Fun(\mcl{C}_0,\mbb{Z}\text{-mod})$ is essentially small. We leave the details to the interested reader.
\end{remark}}

\begin{remark}\label{rem:C3}
Condition (C3) says that $\mcl{B}$ is Grothendieck abelian, save for the possible absence of a generator. Practically speaking, filtered colimits in $\mcl{B}$ are exact whenever $\mcl{B}$ admits a cocontinuous, faithful, exact functor to a Grothendieck abelian category $\mcl{L}$.  We can take for example $\mcl{L}$ equal to $\opn{Vect}$, $R$-mod for a ring $R$, or $\opn{Qcoh}(X)$ for a scheme $X$.
\end{remark}

\subsection{The functor $\opn{Ind}(\mcl{C}_0)\to \mcl{C}$ is a localization}

Throughout the subsection we suppose we are in the setting of Section \ref{sect:setting}, and we adopt the relative expression $\opn{Ind}(\mcl{C}_0)=\mcl{C}_{\opn{atlas}}$.

\begin{proposition}\label{prop:348}
For $\opn{Ind}(\mcl{C}_0)^{\opn{red}}$ the full subcategory in $\opn{Ind}(\mcl{C}_0)$ spanned by objects with reduced atlases (see Definition \ref{def:reduced}), the forgetful functor $\pi:\opn{Ind}(\mcl{C}_0)\to \mcl{C}$ restricts to an equivalence
\[
\opn{Ind}(\mcl{C}_0)^{\opn{red}}\overset{\sim}\to \mcl{C}.
\]
\end{proposition}

\begin{proof}
Let $M_A$ be arbitrary in $\opn{Ind}(\mcl{C}_0)$ and $N_B$ be reduced.  We claim that the map
\[
\pi:\Hom_{\opn{Ind}(\mcl{C}_0)}(M_A,N_B)\to \Hom_{\mcl{C}}(M,N)
\]
is an isomorphism.  We have in both cases, by the universal property of a colimit,
\[
\Hom_{\opn{Ind}(\mcl{C}_0)}(M_A,N_B)=\varprojlim_\alpha \Hom_{\opn{Ind}(\mcl{C}_0)}(M_\alpha,N_B)
\]
and
\[
\Hom_{\mcl{C}}(M,N)=\varprojlim_\alpha\Hom_{\mcl{C}}(M_\alpha,N).
\]
So it suffices to show that the natural map
\[
\Hom_{\opn{Ind}(\mcl{C}_0)}(M_\alpha,N_B)=\varinjlim_\beta\Hom_{\mcl{C}}(M_\alpha,N_\beta)\to \Hom_{\mcl{C}}(M_\alpha,N)
\]
is an isomorphism whenever $M_\alpha$ is in $\mcl{C}_0$.
\par

Now, since the atlas for $N_B$ is reduced, and $\mcl{C}_0$ is stable under quotients, all of the structure maps $N_\beta\to N$ are injective and hence all of the maps
\[
\Hom_{\mcl{C}}(M_\alpha,N_\beta)\to \Hom_{\mcl{C}}(M_\alpha,N)
\]
are injective.  It follows that the map from the directed colimit
\begin{equation}\label{eq:380}
\varinjlim_\beta\Hom_{\mcl{C}}(M_\alpha,N_\beta)\to \Hom_{\mcl{C}}(M_\alpha,N)
\end{equation}
is injective.  The image is simply the subspace of maps $M_\alpha\to N$ which factor through some $N_\beta$.  We claim that finite generation of $M_\alpha$ implies that all maps to $N$ do in fact factor through some $N_\beta$.
\par

For a given map $t:M_\alpha\to N$, we have at each $N_\beta$ the pullback map
\[
t_\beta:L_\beta=N_\beta\times_N M_\alpha\to M_\alpha
\]
which is necessarily injective, and we claim that the corresponding map from the filtered colimit $\varinjlim_\beta L_\beta\to M_\alpha$ is an isomorphism.  Indeed, each pullback $L_\beta\to M_\alpha$ is the kernel of the sequence $M_{\alpha}\to N\to N/N_{\beta}$ and, by exactness of filtered limits in the ambient category $\mcl{B}$, the induced map from the colimit $\varinjlim_{\beta}L_{\beta}\to M_{\alpha}$ is injective with quotient $\varinjlim_{\beta}M/L_\beta\cong M/\varinjlim_\beta L_\beta$.  Such exactness also implies that the map $M/\varinjlim_\beta L_\beta\to N/\varinjlim_\beta N_\beta=N/N=0$ is injective.  In particular, the injection $\varinjlim_{\beta}L_{\beta}\to M_{\alpha}$ is also surjective, and thus an isomorphism, as claimed.
\par

By finite generation of $M_\alpha$, we can now find an index $\gamma$ at which the map $L_\gamma\to M_\alpha$ is an isomorphism.  Equivalently, we can find an index $\gamma$ at which the map $t:M\to N$ factors
\[
\xymatrix{
N_\gamma\ar[r] & N\\
 & M.\ar@{-->}[ul]^{\exists !}\ar[u]_t
}
\]
We now see that the injection \eqref{eq:380} is in fact an isomorphism, and establish bijectivity of the forgetful map
\[
\pi:\Hom_{\opn{Ind}(\mcl{C}_0)}(M_A,N_B)\to \Hom_{\mcl{C}}(M,N)
\]
whenever $N_B$ is reduced.
\par

By the above information we have that the restricted forgetful functor
\[
\opn{Ind}(\mcl{C}_0)^{\opn{red}}\to \mcl{C}
\]
is fully faithful.  For essential surjectivity we note that any object $M$ in $\mcl{C}$ has its canonical atlas, which is reduced, so that $M$ is obtained as the image of $M_{\opn{Can}}$.
\end{proof}

Now, each object $M$ in $\mcl{C}$ has its canonical atlas $M_{\opn{Can}}$, and the operation $M\mapsto M_{\opn{Can}}$ assembles into a functor
\[
\opn{Can}:\mcl{C}\to \opn{Ind}(\mcl{C}_0).
\]
To make things very clear, for any map $\xi:M\to N$ in $\mcl{C}$ and any injection $i_\alpha:M_\alpha\to M$ from an object $M_\alpha$ in $\mcl{C}_0$ we have the injection $\xi_\alpha:M_{\alpha}/\opn{ker}(\xi i_\alpha)\to N$ and diagram
\[
\xymatrix{
M_\alpha\ar[r]^(.35){\xi_\alpha}\ar[d] & M_\alpha/\opn{ker}(\xi i_\alpha)\ar[d]\\
M\ar[r]_{\xi} & N,
}
\]
so that $\xi$ and $i_\alpha$ specify an element in the colimit $\xi_\alpha\in \varinjlim_\beta \Hom_{\mcl{C}}(M_\alpha,N_\beta)$. For each map $M_\alpha\to M_\gamma\to M$ in our canonical system we have a diagram
\[
\xymatrix{
M_\alpha\ar[r]^(.35){\xi_\alpha}\ar[d] & M_\alpha/\opn{ker}(\xi i_\alpha)\ar[d]\\
M_\gamma\ar[r]_(.35){\xi_\gamma} & M_\gamma/\opn{ker}(\xi i_\gamma),
}
\]
so that these $\xi_\alpha$ assemble into an element $\opn{Can}(\xi)\in \Hom_{\opn{Ind}(\mcl{C}_0)}(M_{\opn{Can}},N_{\opn{Can}})$.

\begin{proposition}\label{prop:438}
The canonical atlas functor $\opn{Can}:\mcl{C}\to \opn{Ind}(\mcl{C}_0)$ is a fully faithful right adjoint to the forgetful functor $\pi:\opn{Ind}(\mcl{C}_0)\to \mcl{C}$.  The counit of this adjunction is the identity $id_{N}:\pi(N_{\opn{Can}})\to N$, and the unit of the adjunction
\[
u=u_{M_A}:M_A\to M_{\opn{Can}}
\]
is the unique map in $\opn{Ind}(\mcl{C}_0)$ determined by the diagrams
\[
\xymatrix{
M_\alpha\ar[r]\ar[d]_{f_\alpha} & \bar{M}_\alpha\ar[d]^{\bar{f}_\alpha}\\
M\ar[r]_= & M,
}
\]
where $\bar{M}_\alpha=M_\alpha/\opn{ker}(f_\alpha)$ and $\bar{f}_\alpha$ the inclusion induced by $f_\alpha$.
\end{proposition}

\begin{proof}
As the canonical atlas is reduced, one sees that $\opn{Can}$ has image in the full subcategory of objects in $\opn{Ind}(\mcl{C}_0)$ with reduced atlases, and furthermore provides a section for the forgetful functor
\[
\opn{Ind}(\mcl{C}_0)^{\opn{red}}\to \mcl{C}.
\]
Since this forgetful functor is an equivalence, by Proposition \ref{prop:348}, we see that $\opn{Can}$ is fully faithful.
\par

For each $N_B$ in $\opn{Ind}(\mcl{C}_0)$ we argued in the proof of Proposition \ref{prop:348} that the forgetful map
\[
\Hom_{\opn{Ind}(\mcl{C}_0)}(N_B,M_{\opn{Can}})\to \Hom_{\mcl{C}}(N,M)
\]
is an isomorphism, since the canonical atlas is reduced.  This already gives the proposed adjunction, with counit given by the identity, and gives the unit $M_A\to M_{\opn{Can}}$ as the unique lift of the identity in $\mcl{C}$ to $\opn{Ind}(\mcl{C}_0)$.  This lift is precisely the map $u$ from the statement.
\end{proof}

As a corollary we find that $\mcl{C}$ is in fact a \emph{localization} of the Ind-completion. (We elaborate on this point in Section \ref{sect:frac_w} below).  We consider the collection of morphisms
\begin{equation}\label{eq:W}
W=\big\{
	\text{maps $\xi$ in $\opn{Ind}(\mcl{C}_0)$ for which $\pi(\xi)$ is an isomorphism in }\mcl{C}
\big\},
\end{equation}
and for an arbitrary category $\mcl{D}$ we take
\[
\Fun(\opn{Ind}(\mcl{C}_0),\mcl{D})_W\ \subseteq\ \Fun(\opn{Ind}(\mcl{C}_0),\mcl{D})
\]
the full subcategory of functors which send maps in $W$ to isomorphisms in $\mcl{D}$.

\begin{corollary}\label{cor:470}
For $W$ as above, and arbitrary $\mcl{D}$, restriction along $\pi$ provides an equivalence
\[
\pi^\ast:\Fun(\mcl{C},\mcl{D})\overset{\sim}\to \Fun(\opn{Ind}(\mcl{C}_0),\mcl{D})_W.
\]
\end{corollary}

\begin{proof}
The functor $\opn{Can}:\mcl{C}\to \opn{Ind}(\mcl{C}_0)$ provides the inverse
\[
\opn{Can}^\ast:\Fun(\opn{Ind}(\mcl{C}_0),\mcl{D})_W\to \Fun(\mcl{C},\mcl{D}),\ \ T\mapsto T[W^{-1}]:=T\circ\opn{Can}
\]
to $\pi$.  Indeed, for any functor $T:\opn{Ind}(\mcl{C}_0)\to \mcl{D}$ which inverts the maps in $W$ the localized functor $T[W^{-1}]$ fits into a specified 2-diagram
\[
\xymatrix{
	& \mcl{C}\ar[dr]^{T[W^{-1}]}\\
\opn{Ind}(\mcl{C}_0)\ar[ur]^{\pi}\ar[rr]_T & & \mcl{D} 
}
\]
where the required natural isomorphism $T[W^{-1}]\circ \pi\cong T$ is provided by the unit transformation
\[
T(u):T[W^{-1}]\circ \pi=T\circ (\opn{Can}\circ \pi)\overset{\sim}\to T.
\]
Here we note that the unit maps $u:M_A\to M_{\opn{Can}}$ themselves are in $W$, so that $T(u)$ is in fact a natural isomorphism.
\end{proof}

\begin{definition}
For a situation
\[
\mcl{C}_0\subseteq \mcl{C}\subseteq \mcl{B}
\]
as in Section \ref{sect:setting}, we take $W=W_{\mcl{C}}$ in $\opn{Ind}(\mcl{C}_0)$ the class of all maps $\xi$ in the Ind-completion for which $\pi:\opn{Ind}(\mcl{C}_0)\to \mcl{C}$ sends $\xi$ to an isomorphism.  We call $W$ the class of $\mcl{C}$-relative equivalences in $\opn{Ind}(\mcl{C}_0)$.
\end{definition}

In order to gain better control over the localization map $\opn{Ind}(\mcl{C}_0)\to \mcl{C}$, we introduce an intermediate realization of the localization via a calculus of fractions.

\subsection{Localization via a calculus of fractions}
\label{sect:frac_w}

Given a general category $\mcl{E}$ and a class of maps $T$ in $\mcl{E}$, a localization of $\mcl{E}$ relative to $T$ is any pairing of a category $\mcl{E}'$ with a functor $F:\mcl{E}\to \mcl{E}'$ for which restriction along $F$ provides an equivalence
\[
F_{\ast}:\Fun(\mcl{E}',\mcl{D})\overset{\sim}\to \Fun(\mcl{E},\mcl{D})_T.
\]
Here $\mcl{D}$ is arbitrary and, as above, $\Fun(\mcl{E},\mcl{D})_T$ is the full subcategory spanned by functors in $\Fun(\mcl{E},\mcl{D})$ which send maps in $T$ to isomorphisms in $\mcl{D}$.  Of course, Corollary \ref{cor:470} says that the functor $\pi:\opn{Ind}(\mcl{C}_0)\to \mcl{C}$ realizes $\mcl{C}$ as a localization of $\Ind(\mcl{C}_0)$ relative to the class $W$.
\par

As with the Ind-completion, some models for the localization are better than others.  Below we provide an advantageous realization of the localization $\opn{Ind}(\mcl{C}_0)[W^{-1}]$ via a \emph{calculus of fractions} (see Proposition \ref{prop:frac_universal}).

\begin{remark}
The calculus of fractions construction given below deviates from that of standard references due to certain issues with co/equalizers (cf.\ \cite[\href{https://stacks.math.columbia.edu/tag/04VC}{04VC}]{stacks}).  The methods employed below are applicable to localizations which appear via reflexive subcategories, rather than those which appear via classes of maps which satisfy a certain Ore localizing condition. 
\end{remark}

To begin, for morphisms in our proposed calculus we consider equivalence classes of pairs
\[
s^{-1}f:M_A\overset{f}\to N'_C\overset{s}\leftarrow N_B
\]
in which $s$ is a $\mcl{C}$-relative equivalence.  Each such pair has a uniquely determined map $f_s:M_A\to N_{\opn{Can}}$ which fits into a diagram
\begin{equation}\label{eq:518}
\xymatrix{
	& N'_C\ar[d] & \\
M_A\ar[r]_{f_s}\ar[ur]^f & N_{\opn{Can}}& N_B\ar[l]^u\ar[ul]_s,
}
\end{equation}
and we say two pairs $s^{-1}f,t^{-1}g:M_A\to N_B$ are equivalent if $f_s=g_t$.  In this case we simply write
\[
s^{-1}f=t^{-1}g.
\]

\begin{lemma}\label{lem:530}
There is an equivalence of fractions $s^{-1}f=t^{-1}g$ if and only if, after applying the forgetful functor $\pi:\opn{Ind}(\mcl{C}_0)\to \mcl{C}$, we have $\pi(s)^{-1}\pi(f)=\pi(t)^{-1}\pi(g)$.
\end{lemma}

\begin{proof}
The map
\[
\pi:\Hom_{\opn{Ind}(\mcl{C}_0)}(M_A,N_{\opn{Can}})\to \Hom_{\mcl{C}}(M,N)
\]
is an isomorphism, since $\opn{Can}$ is right adjoint to $\pi$ with counit transformations given by the identity.  Applying $\pi$ to the diagram \eqref{eq:518} tells us $\pi(f_s)=\pi(s)^{-1}\pi(f)$.  Hence $f_s=g_t$ if and only if $\pi(f_s)=\pi(g_t)$ if and only if $\pi(s)^{-1}\pi(f)=\pi(t)^{-1}\pi(g)$.
\end{proof}

We define the set
\[
\Hom_{\mcl{F}_W}(M_A,N_B)=\left\{
\begin{array}{c}
\text{The collection of fractions }\\
s^{-1}f:M_A\to N_B,\ \text{taken up to equivalence}
\end{array}\right\}.
\]
By Lemma \ref{lem:530} the forgetful functor $\pi:\opn{Ind}(\mcl{C}_0)\to \mcl{C}$ induces a bijection
\begin{equation}\label{eq:}
\pi:\Hom_{\mcl{F}_W}(M_A,N_B) \overset{\sim}\to \Hom_{\mcl{C}}(M,N).
\end{equation}

\begin{lemma}\label{lem:703}
Any pair of maps $N'_C\overset{s}\leftarrow N_B\overset{h}\to L_D$ in which $s$ is a $\mcl{C}$-relative equivalence fits into a diagram
\[
\xymatrix{
	& L'_E\\
N'_C\ar[ur]^g &  N_B\ar[l]^s\ar[r]_h &  L_D\ar[ul]_t
}
\]
in which $t$ is a $\mcl{C}$-relative equivalence.
\end{lemma}

\begin{proof}
Take for example $L'_E=L_{\opn{Can}}$, $t$ the unit of adjunction $t=u_{L_D}$, and $g$ the composite
\[
g=\opn{Can}(\pi(h))\opn{Can}(\pi(s))^{-1}u_{N'_C}:N'_C\to N'_{\opn{Can}}\cong N_{\opn{Can}}\to L_{\opn{Can}}.
\]
\end{proof}

Applying $\pi$ so a square as in Lemma \ref{lem:703} gives a square
\[
\xymatrix{
& L'\\
N'\ar[ur]^{\pi(g)} &  N\ar[l]^{\pi(s)}\ar[r]_{\pi(h)} &  L\ar[ul]_{\pi(t)}
}
\]
in which the $\pi(s)$ and $\pi(t)$ are isomorphisms.  Hence $\pi(g)$ is fixed as
\[
\pi(g) = \pi(t)\pi(h)\pi(s)^{-1},\ \ \text{which gives}\ \ \pi(t)^{-1}\pi(g)=\pi(h)\pi(s)^{-1}.
\]
We now define the composition operation
\[
\circ:\Hom_{\mcl{F}_W}(N_B,L_C)\times\Hom_{\mcl{F}_W}(M_A,N_B)\to \Hom_{\mcl{F}_W}(M_A,L_C)
\]
by completing diagrams
\[
r^{-1}h\circ s^{-1}f=
\xymatrixrowsep{3mm}
\xymatrix{
	& & L''_E \\
	&N'_C\ar[ur]^g & & L'_D\ar[ul]_t\ar@{..>}[dl]^h\\
M_A\ar[ur]^f	& & N_B\ar@{..>}[ul]^s & & L_C\ar[ul]_{r}.
}
\]
To see that this composition is well-defined we apply $\pi$ to get
\[
\begin{array}{l}
\pi(r^{-1}h\circ s^{-1}f)=\pi(r)^{-1}\pi(t)^{-1}\pi(g)\pi(f)\vspace{1.5mm}\\
\hspace{2cm}=\pi(r)^{-1}\pi(h)\pi(s)^{-1}\pi(f)=\pi(r^{-1}h)\circ\pi(s^{-1}f).
\end{array}
\]
This identification also tells us that our composition operation is associative.

\begin{definition}
In the setting of Section \ref{sect:setting}, we define the category
\[
\Frac_W=\Frac(\opn{Ind}(\mcl{C}_0),W)
\]
by taking the objects from $\opn{Ind}(\mcl{C}_0)$ and morphisms given by the above fractions
\[
\Hom_{\Frac_W}(M_A,N_B):=\Hom_{\mcl{F}_W}(M_A,N_B).
\]
We consider $\Frac_W$ along with the canonical functor $\opn{loc}:\opn{Ind}(\mcl{C}_0)\to \Frac_W$, $M_A\mapsto M_A$, $f\mapsto id^{-1}f$.
\end{definition}

We have the following strict universal property for our category of fractions.

\begin{proposition}\label{prop:frac_universal}
Given any category $\mcl{D}$, restriction along $\opn{loc}$ provides an equivalence
\[
\opn{loc}^\ast:\Fun(\Frac_W,\mcl{D})\overset{\sim}\to \Fun(\opn{Ind}(\mcl{C}_0),\mcl{D})_W
\]
which is furthermore an isofibration.  Indeed, $\opn{loc}^\ast$ is bijective on both objects and morphisms, and admits a strict inverse.
\end{proposition}

\begin{proof}[Construction]
We have the strict inverse
\[
(\opn{loc}^\ast)^{-1}:\Fun(\opn{Ind}(\mcl{C}_0),\mcl{D})_W\to \Fun(\Frac_W,\mcl{D})
\]
given by taking a functor $F:\opn{Ind}(\mcl{C}_0)\to \mcl{D}$ to the extension $F_W:\Frac_W\to \mcl{D}$ defined by
\[
F_W(M_A)=F(M_A)\ \ \text{and}\ \ F_W(s^{-1}f)=F(s)^{-1}F(f).
\]
For any transformation $\zeta:F\to G$ the maps $\zeta_{M_A}:F(M_A)\to G(M_A)$ are seen to define a transformation $\zeta_W:F_W\to G_W$ due to the diagrams
\[
\xymatrixcolsep{15mm}
\xymatrix{
F(M_A)\ar[r]^{F(f)}\ar[d]_\zeta & F(N'_C)\ar[d]_\zeta & F(N_B)\ar[l]_{F(s)}\ar[d]_{\zeta}\\ 
G(M_A)\ar[r]_{G(f)} & G(N'_C) & G(N_B).\ar[l]^{G(s)} 
}
\]
\end{proof}

Taking this particular fractions model of the localization, Corollary \ref{cor:470} can be reinterpreted as the following.

\begin{corollary}
The unique functor $\pi_W:\Frac_W\to \mcl{C}$ which completes a strictly commuting diagram
\begin{equation}\label{eq:671}
\xymatrix{
	& \Frac_W\ar[dr]^{\pi_W} \\
\opn{Ind}(\mcl{C}_0)\ar[ur]^{\opn{loc}}\ar[rr]_\pi & & \mcl{C}
}
\end{equation}
is an equivalence.
\end{corollary}

To recall again, this map $\pi_W$ is defined on objects by $M_A\mapsto \pi(M_A)=M$ and on morphisms by $s^{-1}f\mapsto \pi(s)^{-1}\pi(f)$.  The (non-strict) inverse to $\pi_W$ is given by taking canonical atlas, and we denote this functor
\[
\opn{Can}_W:\mcl{C}\overset{\sim}\to \Frac_W.
\]

\section{Monoidality of the localization}\label{sec:monoidality}

We've seen that the inclusion $\mcl{C}_0\to \mcl{C}$ realizes $\mcl{C}$ as a localization of the Ind-completion $\opn{Ind}(\mcl{C}_0)$ along the class of ``$\mcl{C}$-relative equivalences.''  When $\mcl{C}_0$ admits a (braided) monoidal structure, we have also seen that this (braided) monoidal structure extends uniquely to a monoidal structure on $\opn{Ind}(\mcl{C}_0)$. In this section we show that the product on $\opn{Ind}(\mcl{C}_0)$ induces a monoidal structure on the localization $\mcl{C}$ whenever the original product functor on $\mcl{C}_0$ is as cocontinuous as possible, so to speak.  In total, we observe a unique extension of the (braided) monoidal structure on $\mcl{C}_0$ to a (braided) monoidal structure on $\mcl{C}$ in this case.

\subsection{The updated setting}\label{sect:update}

We again consider a sequence of full subcategories
\[
\mcl{C}_0\subseteq\mcl{C}\subseteq \mcl{B}
\]
as in Section \ref{sect:setting}.  To recall, $\mcl{B}$ is cocomplete abelian, $\mcl{C}_0$ is essentially small and finitely cocomplete, and $\mcl{C}$ is generated by $\mcl{C}_0$ under filtered colimits.  We also require $\mcl{C}_0$ to be closed under quotients in $\mcl{B}$, and to consist entirely of finitely generated objects.
\par

We suppose now that $\mcl{C}_0$ admits a (braided) monoidal structure $\ot$ which is as cocomplete as possible.  Namely, the tensor product should be finitely cocontinuous in each factor, and we assume the following:
\begin{itemize}
\item[(T)] For any filtered colimit $\varinjlim_\alpha M_\alpha=M$ which occurs in $\mcl{C}_0$, and any $N$ in $\mcl{C}_0$, the natural maps
\[
\varinjlim_\alpha (M_\alpha\ot N)\to M\ot N\ \ \text{and}\ \ \varinjlim_\alpha (N\ot M_\alpha)\to N\ot M
\]
are isomorphisms.
\end{itemize}

\begin{remark}
The colimits referenced in (T) are calculated \emph{in} $\mcl{C}$. We note that the inclusion $\mcl{C}_0\to \mcl{C}$ does not preserve non-finite colimits in general.
\end{remark}

\subsection{Monoidality of the localization \texorpdfstring{$\Frac_W$}{FracW}}\label{subsec:monoidality}

\begin{definition}
Given a monoidal category $\mcl{D}$ and a class of maps $T$, we say that $T$ is stable under the product on $\mcl{D}$ if, for each morphism $s:X\to X'$ in $T$ and object $Y$ in $\mcl{D}$, the maps
\[
s\ot id_Y:X\ot Y\to X'\ot Y\ \ \text{and}\ \ id_Y\ot s:Y\ot X\to Y\ot X'
\] 
are also in $T$.
\end{definition}

To recall, any monoidal structure on $\mcl{C}_0$ which commutes with finite colimits in each factor determines a unique monoidal structure on the Ind-category $\opn{Ind}(\mcl{C}_0)$ which is cocontinuous in each factor, and for which the inclusion $i:\mcl{C}_0\to \opn{Ind}(\mcl{C}_0)$ is monoidal with trivial tensor compatibility (Proposition \ref{prop:ind_product}).
\par

Our main result for the section is the following.

\begin{proposition}\label{prop:loc_tensor}
In the setting of Section \ref{sect:update}, the class $W$ of $\mcl{C}$-relative equivalences in $\opn{Ind}(\mcl{C}_0)$ is stable under the product on $\opn{Ind}(\mcl{C}_0)$.
\end{proposition}

We prove the proposition below, but let us record our primary consequence of interest.  As one imagines, stability of a given class under the monoidal structure tells us that the localization inherits a monoidal structure from that of the original category.

\begin{corollary}\label{cor:854}
In the setting of Section \ref{sect:update}, there is a unique (braided) monoidal structure on the localization $\Frac_W$ under which the structure map $\opn{loc}:\opn{Ind}(\mcl{C}_0)\to \Frac_W$ is a (braided) monoidal functor with trivial structure maps
\[
\opn{loc}(M_A)\ot \opn{loc}(N_B)\overset{=}\to \opn{loc}(M_A\ot N_B).
\]
\end{corollary}

\begin{proof}
In $\Frac_W$, define the product of objects via the product $M_A\ot N_B$ in $\opn{Ind}(\mcl{C}_0)$.  On morphisms we take
\[
(s^{-1}f)\ot (t^{-1}g):=(s\ot t)^{-1}(f\ot g).
\]
Stability of the product under equivalence in the left-hand factor follows from the diagram
\[
\xymatrix{
	& N'_C\ot L_E\ar[dr] & & N''_D\ot L_E\ar[dl]\\
M_A\ot L_E\ar[rr]_{f_s\ot 1=h_r\ot 1}\ar[ur]^{f\ot 1}\ar[urrr]^(.8){h\ot 1}\ar@{..>}[drr] & & N_{\opn{Can}}\ot L_E\ar@{..>}[d] & & N_B\ot L_E\ar[ll]^{u\ot 1}\ar@{..>}[dll]\ar[ul]_{r\ot 1}\ar[ulll]_(.8){s\ot 1}\\
	& & \opn{Can}(N_B\ot L_E) & & 
}
\]
at arbitrary $L_E$.  (Here we've assumed $s^{-1}f=r^{-1}h$.)  Stability under equivalence in the right-hand factor is observed similarly.  The associator, unit structures, braiding, etc.\ are all inherited directly from $\opn{Ind}(\mcl{C}_0)$.
\end{proof}

\begin{remark}
For an alternate proof of Corollary \ref{cor:854}, one can simply employ the strict universal property from Proposition \ref{prop:frac_universal} and proceed as in the proof of Proposition \ref{prop:ind_product}.
\end{remark}

We now prove the main proposition.

\begin{proof}[Proof of Proposition \ref{prop:loc_tensor}]
Consider $M_A$ and $N_B$ in $\opn{Ind}(\mcl{C}_0)$.  We have the unit map $u:M_A\to M_{\opn{Can}}$ which is a $\mcl{C}$-relative equivalence, and we propose that the induced map
\[
M_A\ot N_B\to M_{\opn{Can}}\ot N_B
\]
is also a relative equivalence.  We consider the ``reduction" $M_A\to M_{\mfk{A}}$ defined by replacing the given atlas $\varinjlim_\alpha M_\alpha = M$ with the reduced atlas $\varinjlim_\alpha \bar{M}_\alpha = M$, where each $\bar{M}_\alpha=M_\alpha/K_\alpha$ for $K_\alpha$ the kernel of the structural map $M_\alpha\to M$.  The map $M_A\to M_{\mfk{A}}$ here is the apparent one, i.e.\ the one induced by the projections $M_\alpha\to \bar{M}_\alpha$ at each index $\alpha$.

The unit map for $M_A$ now factors
\[
M_A\to M_{\mfk{A}}\to M_{\opn{Can}}
\]
and, since $M_{\mfk{A}}$ is reduced, the map $M_{\mfk{A}}\to M_{\opn{Can}}$ is an isomorphism in $\opn{Ind}(\mcl{C}_0)$.  This follows by Proposition \ref{prop:348}.  Hence the map $M_A\ot N_B\to M_{\opn{Can}}\ot N_B$ is a relative equivalence if and only if the map $M_A\ot N_B\to M_{\mfk{A}}\ot N_B$ is a relative equivalence.
\par

For the latter claim, by cocontinuity of the product on $\opn{Ind}(\mcl{C}_0)$ in each factor, cocontinuity of the functor $\pi$, and the fact the $\pi|_{\mcl{C}_0}=id$, we have expressions
\[
\varinjlim_{\alpha,\beta}M_\alpha\ot N_\beta=\pi(M_A\ot N_B)\ \ \text{and}\ \ \varinjlim_{\alpha,\beta}\bar{M}_\alpha\ot N_\beta = \pi(M_{\mfk{A}}\ot N_B)
\]
in $\mcl{C}$.  We compare these two colimits.
\par

Consider the filtered indexing category $K$ for $M_A:K\star\{\infty\}\to \mcl{C}$, and at a fixed index $\alpha$, take the undercategory $K_\alpha=K_{\alpha/}$.  Then $K_\alpha$ is filtered and we have the colimit diagram
\[
\bar{M}_\alpha=\varinjlim_\gamma \bar{M}_{\alpha,\gamma}
\]
where we index over objects $\gamma:\alpha\to \alpha'$ in $K_\alpha$ and $\bar{M}_{\alpha,\gamma}=M_\alpha/\opn{ker}(f_\gamma)$ for $f_\gamma=M_A(\gamma):M_\alpha\to M_{\alpha'}$.  Thus each structure map
\[
M_\alpha\ot N_\beta \to \pi(M_A\ot N_B)
\]
at fixed $\alpha$ and $\beta$ factor through the colimit
\[
\varinjlim_\gamma (\bar{M}_{\alpha,\gamma}\ot N_\beta)=(\varinjlim_\gamma \bar{M}_{\alpha,\gamma})\ot N_\beta =\bar{M}_\alpha\ot N_\beta.
\]
Here the middle equality follows by our assumption (T) from Section \ref{sect:update}.  This gives now
\[
\varinjlim_{\alpha,\beta}M_\alpha\ot N_\beta=\varinjlim_{\alpha,\beta}\bar{M}_\alpha\ot N_\beta=\pi(M_A\ot N_B).
\]
From the diagrams
\[
\xymatrix{
M_\alpha\ot N_\beta\ar[r]\ar[dr] & \bar{M}_\alpha\ot N_\beta\ar[d]_=\ar[r] & \pi(M_A\ot N_B)\ar[d]\\
	& \bar{M}_\alpha\ot N_\beta\ar[r] & \pi(M_{\mfk{A}}\ot N_B)
}
\]
at each $\alpha$ and $\beta$ we conclude now that the map
\[
\pi(M_A\ot N_B)\to \pi(M_{\mfk{A}}\ot N_B)
\]
is in fact an isomorphism.  Hence $M_A\ot N_B\to M_{\mfk{A}}\ot N_B$ is a $\mcl{C}$-relative equivalence, as is the map $u\ot id:M_A\ot N_B\to M_{\opn{Can}}\ot N_B$, as desired.
\par

Similar arguments show that the map $id\ot u:M_A\ot N_B\to M_A\ot N_{\opn{Can}}$ is a $\mcl{C}$-relative equivalence at arbitrary $M_A$ and $N_B$.  Since relative equivalences are stable under composition we find that each map
\[
u_{M_A}\ot u_{N_B}:M_A\ot N_B\to M_{\opn{Can}}\ot N_{\opn{Can}}
\]
is a relative equivalence.  Now, for generic relative equivalences $s:M_A\to M'_{C}$ and $t:N_B\to N'_D$ we have a diagram
\[
\xymatrix{
M_A\ot N_B\ar[rrr]^{s\ot t}\ar[d]_{u\ot u} & & & M'_C\ot N'_D\ar[d]^{u\ot u}\\
M_{\opn{Can}}\ot N_{\opn{Can}}\ar[rrr]_{\opn{Can}(\pi s)\ot \opn{Can}(\pi t)} & & & M'_{\opn{Can}}\ot N'_{\opn{Can}}
}
\]
in which the bottom map and vertical maps are $\mcl{C}$-relative equivalences.  Indeed, the bottom map is an isomorphism.  It follows that the product $s\ot t$ is a relative equivalence, and we conclude that $W$ is in fact stable under the product on $\opn{Ind}(\mcl{C}_0)$.
\end{proof}

\subsection{Monoidality for the internal cocompletion \texorpdfstring{$\mcl{C}$}{C}}

\begin{theorem}\label{thm:abs_v1}
In the setting of Section \ref{sect:update}, there is a unique (braided) monoidal structure on $\mcl{C}$ so that the following hold:
\begin{enumerate}
\item[(a)] The inclusion $i:\mcl{C}_0\to \mcl{C}$ is (braided) monoidal with trivial tensor compatibility $i(M)\ot i(N)\overset{=}\to i(M\ot N)$.\vspace{1mm}
\item[(b)] The product functor $\ot:\mcl{C}\times \mcl{C}\to \mcl{C}$ is cocontinuous in each factor.
\end{enumerate}
\end{theorem}

\begin{proof}
We cover the non-braided case.  To deal with the braiding one simply includes the words ``braiding" or ``braided" strategically throughout.  Uniqueness (up to unique isomorphism) is clear.  So we need only establish existence.
\par

We have the inverse
\[
\opn{Can}_W:\mcl{C}\to \Frac_W
\]
to the forgetful functor $\pi_W:\Frac_W\to \mcl{C}$.  For any object $M$ in $\mcl{C}_0$, the inclusion $\mcl{C}_0\to \Frac_W$ sends $M$ to $M$ with the trivial atlas $id_M:\Delta^1\to \mcl{C}_0\subseteq \mcl{C}$.  Since this atlas is reduced, we understand that the unit map $M\to M_{\opn{Can}}$ is a relative equivalence in $\opn{Ind}(\mcl{C}_0)$, and thus an isomorphism in $\Frac_W$.  Hence the unit transformation provides a natural isomorphism which completes a diagram
\[
\xymatrix{
	& \mcl{C}\ar[dr]^{\opn{Can}}\\
\mcl{C}_0\ar[ur]^{i}\ar[rr]_{\opn{incl}} & & \opn{Ind}(\mcl{C}_0). 
}
\]
Since the restriction functor
\[
i^\ast:\Fun(\mcl{C},\opn{Ind}(\mcl{C}_0))\to \Fun(\mcl{C}_0,\opn{Ind}(\mcl{C}_0))
\]
is an isofibration (Proposition \ref{prop:inj_isofib}), we can replace $\opn{Can}$ with an isomorphic functor $\opn{Can}':\mcl{C}\to\opn{Ind}(\mcl{C}_0)$ for which $\opn{Can}'|_{\mcl{C}_0}$ recovers the given inclusion $\mcl{C}_0\to \opn{Ind}(\mcl{C}_0)$.
\par

Take $\opn{Can}'_W=\opn{loc}\circ \opn{Can}':\mcl{C}\to \Frac_W$.  We note that $\opn{Can}'_W$ is isomorphic to $\opn{Can}_W$, and thus is also an inverse to $\pi_W$. There exists in this case a unique natural isomorphism
\[
\zeta:id_{\Frac_W}\overset{\sim}\to\opn{Can}'_W\pi_W
\]
which restrict to the identity on $\mcl{C}_0$.  For this final point one can use the fact that the restriction functor
\[
\Fun(\Frac_W,\Frac_W)\to \Fun(\mcl{C}_0,\Frac_W)
\]
is an isofibration and an equivalence onto the full subcategory of functors whose Ind-extensions send $\mcl{C}$-relative equivalences to isomorphisms (Proposition \ref{prop:frac_universal}).
\par

We now transfer the monoidal structure from $\Frac_W$ to $\mcl{C}$ via the mutually inverse equivalences $\opn{Can}'_W$ and $\pi_W$,
\[
\ot:\mcl{C}\times\mcl{C}\to \mcl{C},\ \ (M,N)\mapsto \pi_W(\opn{Can}'_W(M)\ot\opn{Can}'_W(N))=:M\ot N.
\]
The transformation $\zeta$ provides natural isomorphisms
\[
\begin{array}{l}
\zeta_{M,N}:\vspace{1mm}\opn{Can}'_W(M)\ot\opn{Can}'_W(N)\\\hspace{8em}\overset{\sim}\to
\opn{Can}'_W\pi_W(\opn{Can}'_W(M)\ot\opn{Can}'_W(N))=\opn{Can}'_W(M\ot N).
\end{array}
\]
Via the $\zeta_{M,N}$ and full faithfulness of $\opn{Can}'_W$, the associator and unit structures on $\Frac_W$ induce a unique associator and unit structure for the product $\ot$ on $\mcl{C}$ under which $(\opn{Can}'_W,\zeta_{-,-})$ is a monoidal functor.
\par

Recall that we have chosen $\opn{Can}'_W$ to be the tautological inclusion on $\mcl{C}_0$, and $\zeta$ to be the identity on $\mcl{C}_0$, so that for products of objects in $\mcl{C}_0$ we have
\[
(\text{product in }\mcl{C})\ M\ot N=\pi_W(\opn{Can}'_W(M)\ot \opn{Can}'_W(N))=\pi_W(\opn{incl}(M)\ot\opn{incl}(N))
\]
\[
=\pi_W(\opn{incl}(M\ot N))=M\ot N\ (\text{product in }\mcl{C}_0).
\]
Furthermore, the associators etc.\ for $\mcl{C}$ restrict to the pre-existing associators etc.\ on $\mcl{C}_0$.  Hence the inclusion
\[
i:\mcl{C}_0\to \mcl{C}
\]
is monoidal with trivial tensor compatibility.
\par

Finally, for cocontinuity, we note that the inverse $\pi_W:\Frac_W\to \mcl{C}$ inherits a monoidal structure from the monoidal structure on $\opn{Can}'_W$. Under this induced monoidal structure the tensor compatibility $\pi_W(M_A)\ot\pi_W(N_B)\to \pi_W(M_A\ot N_B)$ restricts to the identity on $\mcl{C}_0$.  Via the diagram
\[
\xymatrix{
 & \Frac_W\ar[dr]^{\pi_W} \\
\opn{Ind}(\mcl{C}_0)\ar[ur]^{\opn{loc}}\ar[rr]_\pi & & \mcl{C}
}
\]
and monoidality of $\opn{loc}$ we see that the forgetful functor
\[
\pi:\opn{Ind}(\mcl{C}_0)\to \mcl{C}.
\]
is canonically monoidal, and also cocontinuous, with $\pi|_{\mcl{C}_0}=id_{\mcl{C}_0}$ as a monoidal functor.
\par

To see that the tensor structure on $\mcl{C}$ is cocontinuous in each factor, for any colimit diagram $\Lambda^M:K\star\{\infty\}\to \mcl{C}$ with cone point $M$ the underlying diagram $\Lambda:K\to \mcl{C}$ lifts to a diagram $\opn{Can}\circ \Lambda:K\to \opn{Ind}(\mcl{C}_0)$ and we take the colimit $M_A$ in $\opn{Ind}(\mcl{C}_0)$.  We may assume the underlying object for this colimit is in fact $M$, by cocontinuity of $\pi$.  Then we have
\[
\opn{colim}_\alpha M_\alpha =M\ \ \text{and}\ \ \opn{colim}_\alpha \opn{Can}(M_\alpha)=M_A
\]
in $\mcl{C}$ and $\opn{Ind}(\mcl{C}_0)$ respectively, and for $N_B$ in $\opn{Ind}(\mcl{C}_0)$ the maps
\[
\opn{colim}_\alpha \opn{Can}(M_\alpha)\ot N_B\to M_A\ot N_B\ \ \text{and}\ \ \opn{colim}_\alpha N_B\ot \opn{Can}(M_\alpha)\to N_B\ot M_A
\]
are isomorphisms.  This follows by cocontinuity of the product on $\opn{Ind}(\mcl{C}_0)$ in each factor.  Applying our cocontinuous map $\pi$, and noting the tensor compatibility isomorphism for $\pi$, we find now that the maps
\[
\opn{colim}_\alpha M_\alpha\ot N=\opn{colim}_\alpha \pi\opn{Can}(M_\alpha)\ot \pi(N_B)\to \pi(M_A)\ot \pi(N_B)=M\ot N
\]
and
\[
\opn{colim}_\alpha N\ot M_\alpha=\opn{colim}_\alpha \pi(N_B)\ot \pi\opn{Can}(M_\alpha)\to \pi(N_B)\ot \pi(M_A)=N\ot M
\]
are also isomorphisms.
\end{proof}

As we saw in the proof, for the induced monoidal structure on $\mcl{C}$ there exists a unique monoidal structure on the forgetful map $\pi:\opn{Ind}(\mcl{C}_0)\to \mcl{C}$ so that $\pi|_{\mcl{C}_0}=id_{\mcl{C}_0}$ as a monoidal functor.  The existence of such a monoidal functor implies cocontinuity of the product on $\mcl{C}$.  We record this alternate formulation of the result.

\begin{theorem}\label{thm:abd_v2}
In the setting of Section \ref{sect:update}, there is a unique (braided) monoidal structure on $\mcl{C}$ so that the following hold:
\begin{enumerate}
\item[(a)] The inclusion $i:\mcl{C}_0\to \mcl{C}$ is (braided) monoidal with trivial tensor compatibility $i(M)\ot i(N)\overset{=}\to i(M\ot N)$.\vspace{1mm}
\item[(b)] The forgetful map $\pi:\opn{Ind}(\mcl{C}_0)\to \mcl{C}$ admits a (braided) monoidal structure for which $\pi|_{\mcl{C}_0}=id_{\mcl{C}_0}$ as a (braided) monoidal functor.
\end{enumerate}
\end{theorem}

\section{Application to vertex operator algebras}\label{sec:vertex-alg}

Here, we show that Theorem \ref{thm:abs_v1} applies when $\mcl{C}_0$ is a suitable category of modules for a vertex operator algebra, strengthening the results proved in \cite{CMY-completions}. Thus let $V$ be a vertex operator algebra with Virasoro $L_0$-eigenvalue grading $V=\bigoplus_{n\in\mathbb{Z}} V_{(n)}$. 

\subsection{Generalized \texorpdfstring{$V$}{V}-modules}

We take the abelian category $\mcl{B}$ of Section \ref{subsec:abs-set} to be the category of generalized $V$-modules (see \cite[Definitions 2.9 and 2.12]{HLZ1} for a full definition). In particular, any generalized $V$-module $M$ decomposes into generalized eigenspaces for the Virasoro $L_0$-operator, $M=\bigoplus_{h\in\mathbb{C}} M_{[h]}$, with no restrictions on $\dim M_{[h]}$. The action of $V$ on $M$ is given by the vertex operator map
\begin{align*}
Y_M: V & \longrightarrow (\mathrm{End}\,M)[[x,x^{-1}]]\\
v & \longmapsto Y_M(v,x) =\sum_{n\in\mathbb{Z}} v_n\,x^{-n-1}.
\end{align*}

The category $\mcl{B}$ of generalized $V$-modules admits small coproducts and quotients and thus is cocomplete. In particular, the colimit of a diagram $M_\ast: K\rightarrow\mcl{B}$ for $K$ an essentially small category is given by the usual set-theoretic construction. To be concrete, first observe that without loss of generality, we may assume $K$ is strictly small, since otherwise there is an equivalence $F: K'\rightarrow K$ from a small category $K'$, and the colimit of $M_\ast\circ F$ can be given the structure of a colimit of $M_\ast$. So it is sufficient to construct the colimit of $M_\ast: K\rightarrow\mcl{B}$ for small $K$ as follows.
Set
\begin{equation*}
\mathrm{colim}_\alpha M_\alpha =\bigoplus_{\alpha\in \mathrm{obj}(K)} M_\alpha\bigg/ J,
\end{equation*}
where the $V$-submodule $J$ is linearly spanned by elements of the form
\begin{equation*}
q_\alpha(m_\alpha) -q_{\beta}(M_f(m_\alpha))
\end{equation*}
for $\alpha,\beta\in \mathrm{obj}(K)$, $m_\alpha\in M_\alpha$, and $f\in\mathrm{Hom}_K(\alpha,\beta)$; here $q_\alpha$ is the inclusion of $M_\alpha$ into the direct sum. Then $\mathrm{colim}_\alpha M_\alpha$ is indeed a cocone of $M_\ast$ with structure maps
\begin{align*}
\varphi_\alpha: M_\alpha & \longrightarrow\mathrm{colim}_\alpha M_\alpha\\
m_\alpha &\longmapsto q_\alpha(m_\alpha)+J.
\end{align*}
Moreover, $\mathrm{colim}_\alpha M_\alpha$ is indeed a colimit of $M_\ast$ because if $N$ is any cocone of $M_\ast$, with structure morphisms $\psi_\alpha: M_\alpha\rightarrow N$ for $\alpha\in\mathrm{obj}(K)$, then there is a unique map $\Phi: \opn{colim}_\alpha M_\alpha\rightarrow N$ making the diagram
\begin{equation*}
\xymatrixcolsep{4pc}
\xymatrix{
M_\alpha \ar[rd]^{\psi_\alpha} \ar[d]_{\varphi_\alpha} & \\
\mathrm{colim}_\alpha M_\alpha \ar[r]^(.58){\Phi} & N\\
}
\end{equation*}
commute for all objects $\alpha\in\mathrm{obj}(K)$; this map is uniquely characterized by
\begin{equation*}
\Phi(\varphi_\alpha(m_\alpha)) = \psi_{\alpha}(m_\alpha)
\end{equation*}
for all $\alpha\in \mathrm{obj}(K)$, $m_\alpha\in M_\alpha$.

\subsection{Tensor products}

We now recall the definition of tensor products of generalized $V$-modules relative to full subcategories of $\mcl{B}$. The tensor product operation is defined using intertwining operators, which we define following \cite[Definition 3.10]{HLZ2}:

\begin{definition}\label{def:intw-op}
Given generalized $V$-modules $M_1$, $M_2$, and $M_3$, an \textit{intertwining operator} of type $\binom{M_3}{M_1\,M_2}$ is a linear map
\begin{align*}
\mathcal{Y}: M_1 & \longrightarrow \opn{Hom}_\mathbb{C}(M_2,M_3)[\log x]\lbrace x\rbrace\\
m_1 & \longmapsto \mcl{Y}(m_1,x) =\sum_{h\in\mathbb{C}}\sum_{k\in\mathbb{Z}_{\geq 0}} (m_1)_{h;k}^{\mcl{Y}}\,x^{-h-1}(\log x)^k
\end{align*}
satisfying the following properties:
\begin{enumerate}

\item \textit{Lower truncation}: For all $m_1\in M_1$, $m_2\in M_2$ and $h\in\mathbb{C}$, $(m_1)_{h+n;k}^{\mcl{Y}}m_2 =0$ for $n\in\mathbb{Z}_{\geq 0}$ sufficiently large, independently of $k$. 

\item The \textit{Jacobi identity}: For all $v\in V$ and $m_1\in M_1$,
\begin{align*}
x_0^{-1}\delta\left(\frac{x_1-x_2}{x_0}\right) Y_{M_3}(v,x_1) & \mcl{Y}(m_1,x_2)  -x_0^{-1}\delta\left(\frac{x_2-x_1}{-x_0}\right) \mcl{Y}(m_1,x_2)Y_{M_2}(v,x_1)\\
& =x_2^{-1}\delta\left(\frac{x_1-x_0}{x_2}\right)\mcl{Y}(Y_{M_1}(v,x_0)m_1,x_2).
\end{align*}
Here $\delta(x)=\sum_{n\in\mathbb{Z}} x^n$ is the formal Dirac delta function.

\item The \textit{$L_{-1}$-derivative property}: For $m_1\in M_1$, $\mcl{Y}(L_{-1}m_1,x)=\frac{d}{dx}\mcl{Y}(m_1,x)$, where $L_{-1}$ is the Virasoro $L_{-1}$-operator acting on $M_1$.

\end{enumerate}
\end{definition}

We can now define tensor products in full subcategories of $\mcl{B}$. Since we also need to work with vector space tensor products of generalized $V$-modules, we now use the distinct notations $\otimes_{\mathbb{C}}$ for vector space tensor products and $\boxtimes$ for the $V$-module tensor product operation on a subcategory of $\mcl{B}$.
\begin{definition}\label{def:tens-prod}
Let $\mcl{C}_0$ be a full subcategory of the category $\mcl{B}$ of generalized $V$-modules, and let $M_1, M_2\in\opn{obj}(\mcl{C}_0)$. A \textit{tensor product} of $M_1$ and $M_2$ relative to $\mcl{C}_0$ is a pair $(M_1\boxtimes M_2,\mcl{Y}_{M_1,M_2})$, where $M_1\boxtimes M_2\in\opn{obj}(\mcl{C}_0)$ and $\mcl{Y}_{M_1,M_2}$ is an intertwining operator of type $\binom{M_1\boxtimes M_2}{M_1\,M_2}$, satisfying the following universal property: For any $M_3\in\opn{obj}(\mcl{C}_0)$ and intertwining operator $\mcl{Y}$ of type $\binom{M_3}{M_1\,M_2}$, there is a unique $V$-module map $f: M_1\boxtimes M_2\rightarrow M_3$ such that the diagram
\begin{equation*}
\xymatrix{
M_1\otimes_{\mathbb{C}} M_2 \ar[d]_{\mcl{Y}_{M_1,M_2}} \ar[r]^(.48){\mcl{Y}} & M_3[\log x]\lbrace x\rbrace\\
(M_1\boxtimes M_2)[\log x]\lbrace x\rbrace \ar[ru]_{f} & \\
}
\end{equation*}
commutes. Here the diagonal arrow denotes the obvious extension of $f$ to spaces of formal series with coefficients in $M_1\boxtimes M_2$ and $M_3$.
\end{definition}

If $\mcl{C}_0$ is a full subcategory of $\mcl{B}$ that is closed under tensor products (relative to $\mcl{C}_0$), then tensor products define a functor $\boxtimes: \mcl{C}_0\times\mcl{C}_0\rightarrow\mcl{C}_0$. 
 Indeed, given morphisms $f_1: M_1\rightarrow N_1$ and $f_2: M_2\rightarrow N_2$ in $\mcl{C}_0$, their tensor product is the unique map $f_1\boxtimes f_2: M_1\boxtimes M_2\rightarrow N_1\boxtimes N_2$ (guaranteed by the universal property of the tensor product $M_1\boxtimes M_2$) such that the diagram
\begin{equation*}
\xymatrix{
M_1\otimes_{\mathbb{C}} M_2 \ar[d]^{\mcl{Y}_{M_1,M_2}} \ar[r]^{f_1\otimes_{\mathbb{C}} f_2} & N_1\otimes_{\mathbb{C}} N_2 \ar[d]^{\mcl{Y}_{N_1,N_2}} \\
(M_1\boxtimes M_2)[\log x]\lbrace x\rbrace \ar[r]^{f_1\boxtimes f_2} & (N_1\boxtimes N_2)[\log x]\lbrace x\rbrace\\
}
\end{equation*}
commutes. Note that $\mcl{Y}_{N_1,N_2}\circ(f_1\otimes_{\mathbb{C}} f_2)$ is an intertwining operator of type $\binom{N_1\boxtimes N_2}{M_1\,M_2}$.

\subsection{Cocontinuity of the tensor product and finite generation}

If $\mcl{C}_0$ is a full subcategory of $\mcl{B}$ that is closed under tensor products relative to $\mcl{C}_0$, then we can now show that under mild conditions, tensor products in $\mcl{C}_0$ commute with all colimits that exist in $\mcl{C}_0$. Thus suppose $K$ is an essentially small category and $M_\ast: K\rightarrow\mcl{C}_0$ is a diagram such that $\opn{colim}_\alpha M_\alpha$, which at first is an object of the category $\mcl{B}$ of generalized $V$-modules, actually lies in $\mcl{C}_0$. This occurs, for example, if $\mcl{C}_0$ is finitely cocomplete and $K$ is finite. Then for any $N\in\opn{obj}(\mcl{C}_0)$, we also have
\begin{equation*}
M_\ast\boxtimes N: K\longrightarrow\mcl{C}_0
\end{equation*}
defined on objects by $\alpha \mapsto M_\alpha\boxtimes N$ and on morphisms by $f \mapsto M_f\boxtimes id_N$. We also have $\opn{colim}_\alpha (M_\alpha\boxtimes N)$, an object of $\mcl{B}$. For $\alpha\in \opn{obj}(K)$, let 
\begin{align*}
\varphi_\alpha: M_\alpha\boxtimes N & \longrightarrow\opn{colim}_\alpha(M_\alpha\boxtimes N)\\
\psi_\alpha: M_\alpha & \longrightarrow \opn{colim}_\alpha M_\alpha
\end{align*}
be the structure maps of the respective colimits. Since the maps $\psi_\alpha$ are morphisms in $\mcl{C}_0$ by assumption, $\psi_\alpha\boxtimes id_N: M_\alpha\boxtimes N\rightarrow (\opn{colim}_\alpha M_\alpha)\boxtimes N$ is defined for each $\alpha$, making $(\opn{colim}_\alpha M_\alpha)\boxtimes N$ a cocone of $M_\ast\boxtimes N$. Thus there is a unique map
\begin{equation*}
\Phi: \opn{colim}_\alpha(M_\alpha\boxtimes N) \longrightarrow (\opn{colim}_\alpha M_\alpha)\boxtimes N
\end{equation*}
such that the diagrams
\begin{equation*}
\xymatrix{
M_\alpha\boxtimes N \ar[d]_{\varphi_\alpha} \ar[rd]^{\psi_\alpha\boxtimes id_N} & \\
\opn{colim}_\alpha(M_\alpha\boxtimes N) \ar[r]_\Phi & (\opn{colim}_\alpha M_\alpha)\boxtimes N\\
}
\end{equation*}
commute for all $\alpha\in\opn{obj}(K)$.

\begin{proposition}\label{prop:VOA-tens-prods-co-cont}
Let $\mcl{C}_0$ be a full subcategory of the category $\mcl{B}$ of generalized $V$-modules that is closed under tensor products relative to $\mcl{C}_0$, and assume that for any intertwining operator $\mcl{Y}$ of type $\binom{M_3}{M_1\,M_2}$ where $M_1$, $M_2\in\opn{obj}(\mcl{C}_0)$ and $M_3\in\opn{obj}(\mcl{B})$, the image of $\mcl{Y}$ is contained in a $\mcl{C}_0$-submodule of $M_3$. Then
the map $\Phi: \opn{colim}_\alpha(M_\alpha\boxtimes N) \rightarrow (\opn{colim}_\alpha M_\alpha)\boxtimes N$  defined above is an isomorphism for any $N\in\opn{obj}(\mcl{C}_0)$ and any diagram $M_\ast:K\rightarrow\mcl{C}_0$ such that $\opn{colim}_\alpha M_\alpha\in\opn{obj}(\mcl{C}_0)$. In particular, $\opn{colim}_\alpha(M_\alpha\boxtimes N)\in\opn{obj}(\mcl{C}_0)$.
\end{proposition}
\begin{proof}
We use the universal property of tensor products in $\mcl{C}_0$ to construct an inverse of $\Phi$. As discussed previously, we may assume without loss of generality that $K$ is small, so that we can realize $\opn{colim}_\alpha M_\alpha$ as
\begin{equation*}
\opn{colim}_\alpha M_\alpha =\bigoplus_{\alpha\in \opn{obj}(K)} M_\alpha\bigg/ J,
\end{equation*}
where $J$ is the $V$-submodule linearly spanned by
\begin{equation*}
q_\alpha(m_\alpha)-q_\beta(M_f(m_\alpha))
\end{equation*}
for $\alpha,\beta\in \opn{obj}(K)$, $m_\alpha\in M_\alpha$, and $f\in\mathrm{Hom}_K(\alpha,\beta)$. As before, $q_\alpha$ is the inclusion of $M_\alpha$ into the direct sum. Also as before, the structure map $\psi_\alpha: M_\alpha\rightarrow\opn{colim}_\alpha M_\alpha$ is given by $\psi_\alpha(m_\alpha)=q_\alpha(m_\alpha)+J$ for $m_\alpha$.
Note that $\opn{colimit}_\alpha M_\alpha$ is spanned by the elements $\psi_\alpha(m_\alpha)$ for $\alpha\in A$, $m_\alpha\in M_\alpha$. We use $\varphi_\alpha$ to denote the structure map $M_\alpha\boxtimes N\rightarrow\opn{colim}_\alpha(M_\alpha\boxtimes N)$ for $\alpha\in\opn{obj}(K)$.

Now we construct an intertwining operator $\mcl{Y}$ of type $\binom{\opn{colim}_\alpha(M_\alpha\boxtimes N)}{\opn{colim}_\alpha M_\alpha\,\, N}$ such that
\begin{align*}
\mcl{Y}\circ(\psi_\alpha\otimes_{\mathbb{C}} id_N) = \varphi_\alpha\circ\mcl{Y}_{M_\alpha,N}: M_\alpha\otimes_{\mathbb{C}} N\longrightarrow\opn{colim}_\alpha(M_\alpha\boxtimes N)[\log x]\lbrace x\rbrace
\end{align*} 
for all $\alpha\in \opn{obj}(K)$. To do so, first note that because coproducts distribute over tensor products in the category of vector spaces, there is a unique linear map
\begin{equation*}
\mcl{X}: \bigg(\bigoplus_{\alpha\in A} M_\alpha\bigg)\otimes_{\mathbb{C}} N \longrightarrow \opn{colim}_\alpha(M_\alpha\boxtimes N)[\log x]\lbrace x\rbrace
\end{equation*}
such that
\begin{equation*}
\mcl{X}\circ(q_\alpha\otimes_{\mathbb{C}} id_N) =\varphi_\alpha\circ\mcl{Y}_{M_\alpha,N}
\end{equation*}
for all $\alpha\in \opn{obj}(K)$. Then because the tensor product of vector spaces is exact, $\mcl{X}$ descends to the desired linear map
\begin{equation*}
\mcl{Y}: (\opn{colim}_\alpha M_\alpha)\otimes_{\mathbb{C}} N\longrightarrow\opn{colim}_\alpha(M_\alpha\boxtimes N)[\log x]\lbrace x\rbrace
\end{equation*}
provided $\mcl{X}\vert_{J\otimes_{\mathbb{C}} N}=0$. Thus we need to show that $\mcl{X}$ annihilates a typical spanning element $(q_\alpha(m_\alpha)-q_\beta(M_\ast(f)(m_\alpha)))\otimes_{\mathbb{C}} n$ of $J\otimes_{\mathbb{C}} N$. Indeed,
\begin{align*}
\mcl{X}(q_\alpha(m_\alpha) & -q_\beta(M_\ast(f)(m_\alpha)),x)n\\ 
& =\varphi_{\alpha}\circ\mcl{Y}_{M_\alpha,N}(m_\alpha,x)n -\varphi_{\beta}\circ\mcl{Y}_{M_\beta,N}(M_\ast(f)(m_\alpha),x)n\\
& =\varphi_{\alpha}\circ\mcl{Y}_{M_\alpha,N}(m_\alpha,x)n -\varphi_{\beta}\circ(M_\ast(f)\boxtimes id_N)\circ\mcl{Y}_{M_\alpha,N}(m_\alpha,x)n\\
& = (\varphi_{\alpha}-\varphi_{\beta}\circ(M_\ast(f)\boxtimes id_N))\circ\mcl{Y}_{M_\alpha,N}(m_\alpha,x)n = 0,
\end{align*}
as required. Thus $\mcl{Y}$ is a well-defined linear map, and it is an intertwining operator since $(\opn{colim}_\alpha M_\alpha)\otimes_{\mathbb{C}} N$ is spanned by the images of $\psi_\alpha\otimes_{\mathbb{C}} id_N$ for $\alpha\in \opn{obj}(K)$, because $\mcl{Y}_{M_\alpha, N}$ is an intertwining operator, and because $\varphi_{\alpha}$ is a $V$-module homomorphism.

By assumption, the image of $\mcl{Y}$ is contained in a $\mcl{C}_0$-submodule of $\opn{colim}_\alpha(M_\alpha\boxtimes N)$, that is, there is a $V$-submodule $\widetilde{M}\subseteq\opn{colim}_\alpha(M_\alpha\boxtimes N)$ such that $\widetilde{M}\in\opn{obj}(\mcl{C}_0)$ and $\mcl{Y}$ factors as
\begin{equation*}
(\opn{colim}_\alpha M_\alpha)\otimes_{\mathbb{C}} N \xrightarrow{\widetilde{\mcl{Y}}} \widetilde{M}[\log x]\lbrace x\rbrace\hookrightarrow\opn{colim}_\alpha(M_\alpha\boxtimes N)[\log x]\lbrace x\rbrace
\end{equation*}
for some intertwining operator $\widetilde{\mcl{Y}}$ of type $\binom{\widetilde{M}}{\opn{colim}_\alpha M_\alpha\,\,N}$.
So by the universal property of tensor products in $\mcl{C}_0$, there is a unique map
\begin{equation*}
\Psi: (\opn{colim}_\alpha M_\alpha)\boxtimes N\longrightarrow \widetilde{M}\hookrightarrow\opn{colim}_\alpha(M_\alpha\boxtimes N)
\end{equation*}
such that the diagram
\begin{equation*}
\xymatrix{
(\opn{colim}_\alpha M_\alpha)\otimes_{\mathbb{C}} N \ar[d]_{\mcl{Y}_{\opn{colim}_\alpha M_\alpha, N}} \ar[r]^(.4){\mcl{Y}} & \opn{colim}_\alpha(M_\alpha\boxtimes N)[\log x]\lbrace x\rbrace\\
((\opn{colim}_\alpha M_\alpha)\boxtimes N)[\log x]\lbrace x\rbrace \ar[ru]_{\Psi} & \\
}
\end{equation*}
commutes. 

To show that $\Psi$ is the inverse of $\Phi$, we first compute
\begin{align*}
\Psi\circ\Phi\circ\varphi_{\alpha}\circ\mcl{Y}_{M_\alpha,N}(m_\alpha,x)n & =\Psi\circ(\psi_{\alpha}\boxtimes id_N)\circ\mcl{Y}_{M_\alpha,N}(m_\alpha,x)n\\
& = \Psi\circ\mathcal{Y}_{\opn{colim}_\alpha M_\alpha, N}(\psi_\alpha(m_\alpha),x)n\\
& =\mcl{Y}(\psi_\alpha(m_\alpha),x)n\\
& =\varphi_{\alpha}\circ\mcl{Y}_{M_\alpha,N}(m_\alpha,x)n
\end{align*}
for all $\alpha\in \opn{obj}(K)$, $m_\alpha\in M_\alpha$, and $n\in N$. Since $\opn{colim}_\alpha(M_\alpha\boxtimes N)$ is spanned by the images of $\varphi_{\alpha}$ by construction, and since $M_\alpha\boxtimes N$ is spanned by the coefficients of powers of $x$ and $\log x$ in $\mcl{Y}_{M_\alpha,N}(m_\alpha,x)n$ as $m_\alpha$ ranges over $M_\alpha$ and $n$ ranges over $N$ (see \cite[Proposition 4.23]{HLZ3}), it follows that $\Psi\circ\Phi$ is the identity on $\opn{colim}_\alpha(M_\alpha\boxtimes N)$. On the other hand,
\begin{align*}
\Phi\circ\Psi\circ\mcl{Y}_{\opn{colim}_\alpha M_\alpha, N}(\psi_\alpha(m_\alpha),x)n & =\Phi\circ\mcl{Y}(\psi_\alpha(m_\alpha),x)n\\
& =\Phi\circ\varphi_{\alpha}\circ\mcl{Y}_{M_\alpha,N}(m_\alpha,x)n\\
& =(\psi_{\alpha}\boxtimes id_N)\circ\mcl{Y}_{M_\alpha,N}(m_\alpha,x)n\\
& =\mcl{Y}_{\opn{colim}_\alpha M_\alpha, N}(\psi_\alpha(m_\alpha),x)n
\end{align*}
for all $\alpha\in \opn{obj}(K)$, $m_\alpha\in M_\alpha$, and $n\in N$. Since $\opn{colim}_\alpha M_\alpha$ is spanned by elements of the form $\psi_\alpha(m_\alpha)$ for $\alpha\in \opn{obj}(K)$ and $m_\alpha\in M_\alpha$, \cite[Proposition 4.23]{HLZ3} again implies that $(\opn{colim}_\alpha M_\alpha)\boxtimes N$ is spanned by coefficients of powers of $x$ and $\log x$ in series of the form $\mcl{Y}_{\opn{colim}_\alpha M_\alpha, N}(\psi_\alpha(m_\alpha),x)n$. It follows that $\Phi\circ\Psi$ is the identity on $(\opn{colim}_\alpha M_\alpha)\boxtimes N$, completing the proof that $\Phi$ is an isomorphism. In particular, $\opn{colim}_\alpha(M_\alpha\boxtimes N)\in\opn{obj}(\mcl{C}_0)$ since it is isomorphic to $(\opn{colim}_\alpha M_\alpha)\boxtimes N\in\opn{obj}(\mcl{C}_0)$.
\end{proof}

\begin{remark}\label{rem:VOA-tens-prod-co-cont}
Under the same conditions on $\mcl{C}_0$, the same proof shows that the natural map $\opn{colim}_\alpha(N\boxtimes M_\alpha)\rightarrow N\boxtimes(\opn{colim}_\alpha M_\alpha)$ is also an isomorphism for any object $N\in\opn{obj}(\mcl{C}_0)$ and any diagram $M_\ast: K\rightarrow\mcl{C}_0$ such that $K$ is essentially small and $\opn{colim}_\alpha M_\alpha\in\opn{obj}(\mcl{C}_0)$.
\end{remark}

Proposition \ref{prop:VOA-tens-prods-co-cont} shows that full subcategories $\mcl{C}_0$ of the category $\mcl{B}$ of generalized $V$-modules satisfy one of the main conditions of Theorem \ref{thm:abs_v1} under mild conditions. For the finite generation condition, we have the following proposition:
\begin{proposition}\label{prop:VOA-fin-gen}
If a generalized $V$-module is finitely generated as a $V$-module, then it is finitely generated in the sense of Definition \ref{def:fg}.
\end{proposition}
\begin{proof}
Let $M$ be a generalized $V$-module with finitely many generators $b_1,\ldots, b_N$ (that is, the smallest $V$-submodule of $M$ that contains all of $b_1,\ldots, b_N$ is $M$ itself). Let $\varinjlim_\alpha M_\alpha$ be any realization of $M$ as a reduced filtered colimit in the category $\mcl{B}$ of generalized $V$-modules. Here, $M=\varinjlim_\alpha M_\alpha$ is the colimit of a diagram $M_\ast: K\rightarrow\mcl{B}$ where $K$ is an essentially small filtered category. Let $\varphi_\alpha: M_\alpha\rightarrow M$ for $\alpha\in \opn{obj}(K)$ be the structure maps of the filtered colimit; they are injective because the filtered colimit is reduced. By construction, $M=\varinjlim_\alpha M_\alpha$ is spanned by the images of $\varphi_\alpha$ for $\alpha\in \opn{obj}(K)$, and because $K$ is filtered, $M$ is also the union of the images of $\varphi_\alpha$ for $\alpha\in \opn{obj}(K)$.

 Now since $M=\bigcup_{\alpha\in A} \mathrm{Im}\,\varphi_\alpha$, there exist $\alpha_1,\ldots,\alpha_N\in \opn{obj}(K)$ such that the generator $b_i$ of $M$ is contained in $\opn{Im}\,\varphi_{\alpha_i}$ for $i=1,\ldots, N$, that is, $b_i=\varphi_{\alpha_i}(m_{i})$ for some $m_i\in M_{\alpha_i}$. Since $K$ is filtered, there is some $\gamma\in \opn{obj}(K)$ such that there exist morphisms $f_i: k_{\alpha_i}\rightarrow k_\gamma$ for all $i=1,\ldots, N$. Then
\begin{equation*}
b_i =\varphi_{\alpha_i}(m_i)=\varphi_\gamma(M_{f_i}(m_i)),
\end{equation*}
so $b_i\in\mathrm{Im}\,\varphi_\gamma$ for all $i$. Since $b_1,\ldots, b_N$ generate $M$ as a $V$-module, it follows that $\opn{Im}\varphi_\gamma =W$ and thus $\varphi_\gamma$ is surjective as well as injective. That is, $\varphi_\gamma$ is an isomorphism, and thus $M$ is finitely generated in the sense of Definition \ref{def:fg}.
\end{proof}

\subsection{Main results for vertex operator algebras}

We are now ready to state the main application of Theorem \ref{thm:abs_v1} to vertex operator algebras:
\begin{theorem}\label{thm:braid-mon-for-C0}
Let $V$ be a vertex operator algebra, let $\mcl{C}_0$ be an essentially small full subcategory of the category $\mcl{B}$ of generalized $V$-modules, and let $\mcl{C}$ be the full subcategory of $\mcl{B}$ whose objects are isomorphic to essentially small filtered colimits of diagrams $M_\ast: K\rightarrow\mcl{C}_0$. Assume also that:
\begin{enumerate}

\item $\mcl{C}_0$ is closed under finite coproducts and arbitrary quotients.

\item $\mcl{C}_0$ admits a braided monoidal structure whose tensor product functor is given by tensor products relative to $\mcl{C}_0$.

\item For any intertwining operator $\mcl{Y}$ of type $\binom{M_3}{M_1\,M_2}$ where $M_1$, $M_2\in\opn{obj}(\mcl{C}_0)$ and $M_3\in\opn{obj}(\mcl{C})$, the image of $\mcl{Y}$ is contained in a $\mcl{C}_0$-submodule of $M_3$.

\item All objects of $\mcl{C}_0$ are finitely generated as generalized $V$-modules.

\end{enumerate}
Then there is a unique braided monoidal structure on $\mcl{C}$ such that:
\begin{enumerate}

\item[(a)] The inclusion $i:\mcl{C}_0\rightarrow\mcl{C}$ is braided monoidal with trivial structure maps $i(M)\boxtimes i(N)\xrightarrow{=} i(M\boxtimes N)$.

\item[(b)] The tensor product functor $\boxtimes: \mcl{C}\times\mcl{C}\rightarrow\mcl{C}$ is cocontinuous in each factor.
\end{enumerate}
\end{theorem}
\begin{proof}
The category $\mcl{C}_0$ is stable under finite colimits in $\mcl{B}$ and satisfies condition (C2) of Section \ref{sect:setting} due to assumption (1). Condition (C1) of Section \ref{sect:setting} follows from assumption (4) and Proposition \ref{prop:VOA-fin-gen}, and in view of Remark \ref{rem:C3}, condition (C3) holds because the forgetful functor from $\mcl{B}$ to $\mathbb{C}$-vector spaces is faithful, exact, and cocontinuous (because a colimit in $\mcl{B}$ is just a $V$-module structure on the colimit of the underlying vector spaces).
The cocontinuity requirements for the tensor product on $\mcl{C}_0$ in Theorem \ref{thm:abs_v1} hold by Proposition \ref{prop:VOA-tens-prods-co-cont} and Remark \ref{rem:VOA-tens-prod-co-cont}, 
and the tensor product on $\mcl{C}_0$ preserves surjections by the proof of \cite[Proposition 4.26]{HLZ3}. Now the result follows from Theorem \ref{thm:abs_v1}.
\end{proof}

As in \cite[Section 6]{CMY-completions}, we need to address whether the braided monoidal structure on $\mcl{C}$ given by Theorem \ref{thm:braid-mon-for-C0} is vertex algebraically natural, that is, whether it agrees with the braided monoidal structure on a category of generalized $V$-modules as described in \cite{HLZ8} (see also \cite[Section 3.3]{CKM}). In particular, not only should the tensor product module $X_1\boxtimes X_2$ for $X_1$, $X_2\in\opn{obj}(\mcl{C})$ be the tensor product relative to $\mcl{C}$ as in Definition \ref{def:tens-prod}, but also the unit object should be $V$, and the unit, associativity, and braiding isomorphisms of $\mcl{C}$ should be as described in \cite{HLZ8}. An obvious necessary condition for this is that the underlying braided monoidal structure on $\mcl{C}_0$ should be that described in \cite{HLZ8}. It turns out this is also sufficient for the braided monoidal structure on $\mcl{C}$ given in Theorem \ref{thm:braid-mon-for-C0} to be the same as that described in \cite{HLZ8}. To prove this, we could apply similar arguments to those of \cite[Section 6]{CMY-completions} to the monoidal structure on $\mcl{C}$ obtained in the proof of Theorem \ref{thm:abs_v1}. But we can avoid repeating such arguments by making the following observations.

Under the extra assumption that $\mcl{C}_0$ is closed under submodules, \cite[Theorem 5.3]{CMY-completions} obtained a braided monoidal structure on $\mcl{C}$ that was shown in \cite[Theorems 6.2 and 6.3]{CMY-completions} to agree with the vertex algebraic braided monoidal structure of \cite{HLZ8}. This braided monoidal structure on $\mcl{C}$ extends that of $\mcl{C}_0$ by \cite[Theorem 5.4]{CMY-completions}, and because its tensor product is that of Definition \ref{def:tens-prod}, the proof of Proposition \ref{prop:VOA-tens-prods-co-cont} shows that it is cocontinuous. Thus if $\mcl{C}_0$ is closed under submodules, the braided monoidal structure on $\mcl{C}$ obtained in \cite{CMY-completions} has to be equivalent to the one given in Theorem \ref{thm:braid-mon-for-C0}. In particular, the braided monoidal structure of Theorem \ref{thm:braid-mon-for-C0} is the vertex algebraic one of \cite{HLZ8} at least when $\mcl{C}_0$ is closed under submodules.

Careful examination of the arguments in \cite{CMY-completions} shows that the only essential use of the assumption that $\mcl{C}_0$ is closed under submodules occurs in \cite[Lemma 4.3]{CMY-completions}, which is then used to prove \cite[Proposition 5.2]{CMY-completions}. 
Fortunately, it turns out that the proof of Proposition \ref{prop:loc_tensor} above provides a proof of \cite[Proposition 5.2]{CMY-completions} without assuming that $\mcl{C}_0$ is closed under submodules.
To explain why, recall that, using the notation of Section \ref{subsec:monoidality} (except that we now use $\boxtimes$ for tensor products in $\opn{Ind}(\mcl{C}_0)$ instead of $\otimes$), the proof of Proposition \ref{prop:loc_tensor} shows that the maps
\begin{align*}
\pi(u_{M_A}\boxtimes id_{N_B}): \pi(M_A\boxtimes N_B)=\varinjlim_{\alpha,\beta} M_\alpha\boxtimes N_\beta & \longrightarrow\pi(M_{\opn{Can}}\boxtimes N_B)\\
\pi(id_{M_A}\boxtimes u_{N_B}): \pi(M_A\boxtimes N_B)=\varinjlim_{\alpha,\beta} M_\alpha\boxtimes N_\beta & \longrightarrow\pi(M_A\boxtimes N_{\opn{Can}})
\end{align*}
are isomorphisms for $M_A,N_B\in\opn{obj}(\opn{Ind}(\mcl{C}_0))$. But \cite[Proposition 5.2]{CMY-completions} is simply the assertion that $\pi(u_{M_A}\boxtimes id_{N_B})$ is an isomorphism in the special case $M_A=(M_1)_{\opn{Can}}\boxtimes (M_2)_{\opn{Can}}$ for $M_1,M_2\in\opn{obj}(\mcl{C})$, and that $\pi(id_{M_A}\boxtimes u_{N_B})$ is an isomorphism in the special case $N_B=(N_1)_{\opn{Can}}\boxtimes (N_2)_{\opn{Can}}$ for $N_1,N_2\in\opn{obj}(\mcl{C})$, so Proposition \ref{prop:loc_tensor} shows that \cite[Proposition 5.2]{CMY-completions} holds in the setting of Theorem \ref{thm:braid-mon-for-C0}.

\begin{remark}\label{rem:M-Can-vs-alpha-M}
We remark that (the equivalent of) $M_{\opn{Can}}$ in \cite{CMY-completions} is defined slightly differently than in Definition \ref{def:M-Can} (where $M_{\opn{Can}}(f_\alpha)=\opn{Im} f_\alpha$ for a map $f_\alpha$ in the essentially small filtered category $K_M$ of injections from $\mcl{C}_0$-objects $M_\alpha$ to $M$).
 In \cite{CMY-completions}, the equivalent of $M_{\opn{Can}}$ for $M\in\opn{obj}(\mcl{C})$ is denoted $\alpha_M$ and is defined to be the functor from the directed set $I_M$ of $\mcl{C}_0$-submodules of $M$ (ordered by inclusion, so that we can view $I_M$ as a small filtered category) to $\mcl{C}_0$ that sends a $\mcl{C}_0$-submodule of $M$ to itself. But there is no essential difference between $M_{\opn{Can}}$ of Definition \ref{def:M-Can} and $\alpha_M$ of \cite{CMY-completions}, because the functor $K_M\rightarrow I_M$ given on objects by $f_\alpha\mapsto\opn{Im} f_\alpha$ is an equivalence, and under this equivalence, $M_{\opn{Can}}$ is identified with $\alpha_M$.
\end{remark}

Since Proposition \ref{prop:loc_tensor} shows that \cite[Proposition 5.2]{CMY-completions} holds even if $\mcl{C}_0$ is not closed under submodules, the remaining arguments of \cite{CMY-completions} show that $\mcl{C}$ admits the vertex algebraic braided monoidal structure of \cite{HLZ8}. As argued above, this braided monoidal structure on $\mcl{C}$ is equivalent to that of Theorem \ref{thm:braid-mon-for-C0}, so we conclude:
\begin{theorem}\label{thm:vertex-mon-for-C0}
Let $V$ be a vertex operator algebra, let $\mcl{C}_0$ be an essentially small full subcategory of the category $\mcl{B}$ of generalized $V$-modules, and let $\mcl{C}$ be the full subcategory of $\mcl{B}$ whose objects are the unions of submodules which are objects of $\mcl{C}_0$. Assume also that:
\begin{enumerate}

\item $\mcl{C}_0$ is closed under finite coproducts and arbitrary quotients.

\item $\mcl{C}_0$ admits the vertex algebraic braided monoidal structure of \cite{HLZ8}, so that in particular $V$ is the unit object of $\mcl{C}_0$ and the tensor product functor of $\mcl{C}_0$ is given by tensor products relative to $\mcl{C}_0$.

\item For any intertwining operator $\mcl{Y}$ of type $\binom{M_3}{M_1\,M_2}$ where $M_1$, $M_2\in\opn{obj}(\mcl{C}_0)$ and $M_3\in\opn{obj}(\mcl{C})$, the image of $\mcl{Y}$ is contained in a $\mcl{C}_0$-submodule of $M_3$.

\item All objects of $\mcl{C}_0$ are finitely generated as generalized $V$-modules.

\end{enumerate}
Then $\mcl{C}$ is additive and cocomplete and admits the vertex algebraic braided monoidal structure of \cite{HLZ8}. Furthermore, the inclusion $\mcl{C}_0\hookrightarrow\mcl{C}$ is braided monoidal and the tensor product $\boxtimes: \mcl{C}\times\mcl{C}\rightarrow\mcl{C}$ is cocontinuous in each factor.
\end{theorem}

\begin{remark}
The assertion in this theorem that $\mcl{C}$ is cocomplete follows from \cite[Proposition 4.3]{CMY-completions}, for example, which shows that $\mcl{C}$ is closed under small coproducts and quotients, and thus also is closed under small colimits. The assertion that the tensor product on $\mcl{C}$ is cocontinuous follows not only from Theorem \ref{thm:abs_v1} but also more directly from Proposition \ref{prop:VOA-tens-prods-co-cont} and Remark \ref{rem:VOA-tens-prod-co-cont}, given that the tensor product on $\mcl{C}$ is the vertex algebraic one of Definition \ref{def:tens-prod}.
\end{remark}

%
%

\subsection{\texorpdfstring{$C_1$}{C1}-cofinite modules}

Now we want to show that Theorem \ref{thm:vertex-mon-for-C0} applies to at least one reasonable category $\mcl{C}_0$ of generalized $V$-modules for any vertex operator algebra $V$. To this end, we say that
a generalized $V$-module $M$ is \textit{lower bounded} if its generalized $L_0$-eigenvalue grading $M=\bigoplus_{h\in\mathbb{C}} M_{[h]}$ has a lower bound, that is, $M_{[h]}=0$ for $\opn{Re}\,h$ sufficiently negative. Then a \textit{$C_1$-cofinite $V$-module} is a lower bounded generalized $V$-module $M$ such that $\dim M/C_1(M)<\infty$, where
\begin{equation*}
C_1(M)=\opn{span}\left\lbrace v_{-1}m\,\,\bigg\vert\,\,v\in\bigoplus_{n=1}^\infty V_{(n)},\,m\in M\right\rbrace.
\end{equation*}
Let $\mcl{C}_0$ be the full subcategory of $\mcl{B}$ consisting of $C_1$-cofinite $V$-modules. Then $\mcl{C}_0$ is a $\mathbb{C}$-linear additive category closed under finite direct sums and quotients, and thus is finitely cocomplete. Every $C_1$-cofinite $V$-module is countable dimensional (see for example \cite[Lemma 6]{Miy-C1}), and thus $\mcl{C}_0$ is also essentially small (any isomorphism class in $\mcl{C}_0$ is characterized by a set of operators $v_n$ for $v\in V$, $n\in\mathbb{Z}$ acting on a fixed countable dimensional vector space).

By \cite[Main Theorem]{Miy-C1}, the category $\mcl{C}_0$ of $C_1$-cofinite $V$-modules is closed under tensor products, and in fact, by \cite[Theorem 1.2]{Hu-C1}, tensor products relative to $\mcl{C}_0$ give $\mcl{C}_0$ the structure of a vertex algebraic braided monoidal category as in \cite{HLZ8}. The category $\mcl{C}_0$ also satisfies the intertwining operator condition (3) of Theorem \ref{thm:vertex-mon-for-C0} by \cite[Key Theorem]{Miy-C1} and \cite[Corollary 2.12]{CMY-completions} (see also \cite{Nahm}). Finally, it is easy to see that every object $M$ of $\mcl{C}_0$ is finitely generated as a generalized $V$-module, for example by homogeneous basis elements of a finite-dimensional homogeneous subspace $X$ such that $M=C_1(M)\oplus X$ as graded vector spaces (see for example \cite[Proposition 2.1]{CMY-completions}). Thus Theorem \ref{thm:vertex-mon-for-C0} yields:

\begin{theorem}\label{thm:braid-mon-for-C1}
Let $\mcl{C}_0$ be the category of $C_1$-cofinite modules for any vertex operator algebra $V$, and let $\mcl{C}$ be the full subcategory of generalized $V$-modules whose objects are the unions of submodules which are objects of $\mcl{C}_0$. Then $\mcl{C}$ is additive and cocomplete, and $\mcl{C}$ admits the vertex algebraic braided monoidal structure of \cite{HLZ8}. Furthermore, the inclusion $\mcl{C}_0\hookrightarrow\mcl{C}$ is braided monoidal and the tensor product $\boxtimes: \mcl{C}\times\mcl{C}\rightarrow\mcl{C}$ is cocontinuous in each factor.
\end{theorem}

\section{Another proof of Theorem \texorpdfstring{\ref{thm:vertex-mon-for-C0}}{6.7}}\label{sec:another-proof}

This section is intended as a supplement to \cite{CMY-completions} and can be read largely independently of the rest of this paper. The purpose is to show how Theorem \ref{thm:vertex-mon-for-C0} can be proved more vertex algebraically by repeating and modifying where necessary the more concrete arguments of \cite{CMY-completions}, rather than using the more abstract category-theoretic arguments of this paper. As discussed in Section \ref{sec:vertex-alg}, Theorem \ref{thm:vertex-mon-for-C0} (without the assertion that the tensor product on $\mcl{C}$ is cocontinuous in each factor) is proved in \cite[Theorem 1.1]{CMY-completions} under the additional assumption that the category $\mcl{C}_0$ of Theorem \ref{thm:vertex-mon-for-C0} is closed under submodules. Thus we just need to replace any arguments in \cite{CMY-completions} that used this extra assumption on $\mcl{C}_0$ with new proofs that do not use this assumption. Then the rest of the proof of Theorem \ref{thm:vertex-mon-for-C0} will follow exactly the arguments of \cite{CMY-completions}, and at the end we use Proposition \ref{prop:VOA-tens-prods-co-cont} to show that the tensor product on $\mcl{C}$ is cocontinuous.

Careful examination shows that the assumption in \cite{CMY-completions} that $\mcl{C}_0$ is closed under submodules was used directly in only two places, in the proofs of \cite[Lemma 4.3]{CMY-completions} and \cite[Proposition 4.3]{CMY-completions}. In the proof of \cite[Proposition 4.3]{CMY-completions}, the assumption was only used to show that $\mcl{C}$ is closed under submodules; this is not a problem here because Theorem \ref{thm:vertex-mon-for-C0} does not claim that $\mcl{C}$ is closed under submodules in general. However, \cite[Lemma 4.3]{CMY-completions} was used in the proof of \cite[Proposition 5.2]{CMY-completions}, which plays a crucial role in the construction of the associativity isomorphisms in $\mcl{C}$ and the proof of the pentagon. Thus it suffices to give a different proof of \cite[Proposition 5.2]{CMY-completions} in the general setting of Theorem \ref{thm:vertex-mon-for-C0}. As mentioned in Section \ref{sec:vertex-alg}, \cite[Proposition 5.2]{CMY-completions} in fact follows from Propositions \ref{prop:loc_tensor} and \ref{prop:VOA-tens-prods-co-cont}, but here we present a more direct proof in the vertex algebraic context that does not require Proposition \ref{prop:VOA-tens-prods-co-cont}.

As was again discussed in Section \ref{sec:vertex-alg}, \cite[Proposition 5.2]{CMY-completions} is a special case of the assertion that certain maps $\Phi_{M_A,N_B}:=\pi(u_{M_A}\boxtimes id_{N_B})$ and $\Psi_{M_A,N_B}:=\pi(id_{M_A}\boxtimes u_{N_B})$ are isomorphisms. We repeat the definition of these maps here, except that for consistency with \cite{CMY-completions}, we slightly modify the definition of the functor $M_{\opn{Can}}$ of Definition \ref{def:M-Can} to agree more literally with the functor denoted $\alpha_M$ in \cite{CMY-completions} (see Remark \ref{rem:M-Can-vs-alpha-M}). To begin, recall that given directed sets $A$ and $B$, which we view as small filtered categories, and given functors $M_A: A\rightarrow\mcl{C}_0$ and $N_B: B\rightarrow\mcl{C}_0$ (which we denote on objects by $\alpha\mapsto M_\alpha$, $\beta\mapsto N_\beta$ for $\alpha\in A$, $\beta\in B$), their tensor product
\begin{equation*}
M_A\boxtimes N_B: A\times B\longrightarrow\mcl{C}_0
\end{equation*}
is given on objects by $(\alpha,\beta)\mapsto M_\alpha\boxtimes N_\beta$ for $\alpha\in A$, $\beta\in B$.

We denote the direct limits $\varinjlim M_A$ and $\varinjlim N_B$ in the category $\mcl{B}$ of generalized $V$-modules by $M$ and $N$, respectively, so $M$ and $N$ are both objects of $\mcl{C}$. We have the functor $M_{\opn{Can}}$ from the directed set of $\mcl{C}_0$-submodules of $M$ (ordered by inclusion) to $\mcl{C}_0$ which sends any $\mcl{C}_0$-submodule of $M$ to itself, and we similarly have the functor $N_{\opn{Can}}$. Let $\varphi_\alpha: M_\alpha\rightarrow\varinjlim M_A$ and $\psi_\beta: N_\beta\rightarrow\varinjlim N_B$ be the structure morphisms of the direct limits for $\alpha\in A$ and $\beta\in B$. Then we have unique maps
\begin{align*}
\Phi_{M_A,N_B}: \varinjlim M_A\boxtimes N_B & \longrightarrow \varinjlim M_{\opn{Can}}\boxtimes N_B\\
\Psi_{M_A,N_B}: \varinjlim M_A\boxtimes N_B & \longrightarrow \varinjlim M_A\boxtimes N_{\opn{Can}}
\end{align*}
such that the diagrams
\begin{equation*}
\xymatrixcolsep{3pc}
\xymatrix{
M_\alpha\boxtimes N_\beta \ar[d]^{\phi_{\alpha,\beta}} \ar[r]^{\varphi_\alpha\boxtimes id_{N_\beta}} & \mathrm{Im}\,\varphi_\alpha\boxtimes N_\beta \ar[d]_{\phi_{\mathrm{Im}\,\varphi_\alpha,\beta}}\\
\varinjlim M_A\boxtimes N_B \ar[r]^(.48){\Phi_{M_A,N_B}} & \varinjlim M_{\opn{Can}}\boxtimes N_B\\
}
\end{equation*}
and
\begin{equation*}
\xymatrixcolsep{3pc}
\xymatrix{
M_\alpha\boxtimes N_\beta \ar[d]^{\phi_{\alpha,\beta}} \ar[r]^{id_{M_\alpha}\boxtimes \psi_\beta} & M_\alpha\boxtimes \mathrm{Im}\,\psi_\beta \ar[d]^{\phi_{\alpha,\opn{Im}\,\psi_\beta}}\\
\varinjlim M_A\boxtimes N_B \ar[r]^(.48){\Psi_{M_A,N_B}} & \varinjlim M_A\boxtimes N_{\opn{Can}}\\
}
\end{equation*}
commute for $\alpha\in A$, $\beta\in B$. Here we use $\phi$ with subscripts to denote the structure maps of the tensor product direct limits.

To see why $\Phi_{M_A,N_B}$ is well defined, let $f_{\alpha_1}^{\alpha_2}:M_{\alpha_1}\rightarrow M_{\alpha_2}$ for $\alpha_1\leq\alpha_2$ in $A$ denote the transition map, that is, the image under $M_A$ of the unique morphism $\alpha_1\rightarrow\alpha_2$ in $A$. Similarly, we use $g_{\beta_1}^{\beta_2}: N_{\beta_1}\rightarrow N_{\beta_2}$ for $\beta_1\leq\beta_2$ in $B$ to denote the corresponding transition map of $N_B$. Since $\varphi_{\alpha_2}\circ f_{\alpha_1}^{\alpha_2}=\varphi_{\alpha_1}$, we have $\mathrm{Im}\,\varphi_{\alpha_1}\subseteq\opn{Im} \varphi_{\alpha_2}$, and thus we also have the transition map $f_{\opn{Im} \varphi_{\alpha_1}}^{\opn{Im}\varphi_{\alpha_2}}: \opn{Im}\varphi_{\alpha_1}\rightarrow\opn{Im}\varphi_{\alpha_2}$ of $M_{\opn{Can}}$, which is just the inclusion and satisfies
\begin{equation*}
f^{\opn{Im}\varphi_{\alpha_2}}_{\opn{Im}\varphi_{\alpha_1}}\circ\varphi_{\alpha_1}=\varphi_{\alpha_2}\circ f_{\alpha_1}^{\alpha_2}.
\end{equation*}
Then for any $\alpha_1\leq\alpha_2$ in $A$ and $\beta_1\leq\beta_2$ in $B$, we have
\begin{align*}
\phi_{\opn{Im}\varphi_{\alpha_2},\beta_2} & \circ(\varphi_{\alpha_2}\boxtimes id_{N_{\beta_2}})\circ(f_{\alpha_1}^{\alpha_2}\boxtimes g_{\beta_1}^{\beta_2})\\
& =\phi_{\mathrm{Im}\,\varphi_{\alpha_2},\beta_2}\circ(f_{\opn{Im}\varphi_{\alpha_1}}^{\opn{Im}\varphi_{\alpha_2}}\boxtimes g_{\beta_1}^{\beta_2})\circ(\varphi_{\alpha_1}\boxtimes id_{N_{\beta_1}})\\
&=\phi_{\mathrm{Im}\,\varphi_{\alpha_1},\beta_1}\circ(\varphi_{\alpha_1} \boxtimes id_{N_{\beta_1}})
\end{align*}
Then the universal property of $\varinjlim M_A\boxtimes N_B$ implies $\Phi_{M_A,N_B}$ exists and is unique. Similarly, $\Psi_{M_A,N_B}$ exists and is unique.

Now as mentioned in Section \ref{sec:vertex-alg}, \cite[Proposition 5.2]{CMY-completions} asserts that $\Phi_{M_A,N_B}$ is an isomorphism in the special case $M_A=(M_1)_{\opn{Can}}\boxtimes (M_2)_{\opn{Can}}$ for $M_1$, $M_2\in\opn{obj}(\mcl{C})$, and that $\Psi_{M_A,N_B}$ is an isomorphism in the special case $N_B=(N_1)_{\opn{Can}}\boxtimes (N_2)_{\opn{Can}}$ for $N_1$, $N_2\in\opn{obj}(\mcl{C})$. Thus Theorem \ref{thm:vertex-mon-for-C0} follows from the arguments of \cite{CMY-completions}, with the proof of \cite[Proposition 5.2]{CMY-completions} replaced by the proof of the following proposition:
\begin{proposition}\label{prop:key-prop}
In the setting of Theorem \ref{thm:vertex-mon-for-C0}. let $M_A: A\rightarrow\mcl{C}_0$ and $N_B: B\rightarrow\mcl{C}_0$ be functors from directed sets $A$ and $B$. Then the $\mcl{C}$-morphisms $\Phi_{M_A,N_B}$ and $\Psi_{M_A,N_B}$ are isomorphisms of generalized $V$-modules.
\end{proposition}

The proof is based on ideas from the proofs of Propositions \ref{prop:348} and \ref{prop:loc_tensor}, but it is more direct in the concrete vertex algebraic context thanks to the following lemma:
\begin{lemma}\label{lem:right-exact}
In the setting of Theorem \ref{thm:vertex-mon-for-C0}, let $f: M_1\rightarrow M_2$ be a surjection in $\mcl{C}_0$ and let $N$ be an object of $\mcl{C}_0$. Then $\opn{Ker}(f\boxtimes id_N)\subseteq M_1\boxtimes N$ is spanned by coefficients of powers of $x$ and $\log x$ in series of the form $\mcl{Y}_{M_1,N}(k,x)n$ for $k\in\opn{Ker} f$ and $n\in N$. Similarly, $\opn{Ker}(id_N\boxtimes f)\subseteq N\boxtimes M_1$ is spanned by coefficients of powers of $x$ and $\log x$ in series of the form $\mcl{Y}_{N,M_1}(n,x)k$ for $k\in\opn{Ker} f$ and $n\in N$.
\end{lemma}
\begin{proof}
This is proved in \cite[Proposition 4.26]{HLZ3}, but we review the proof of the first assertion for the reader's convenience, and the second assertion is proved similarly.

Let $K$ be the subspace of $M_1\boxtimes N$ spanned by coefficients of powers of $x$ and $\log x$ in series of the form $\mcl{Y}_{M_1,N}(k,x)n$ for $k\in\opn{Ker} f$, $n\in N$. Then $K$ is a $V$-submodule by the commutator formula for intertwining operators (which is the coefficient of $x_0^{-1}$ in the Jacobi identity) and the fact that $\opn{Ker} f$ is a submodule. The definition of $f\boxtimes id_N$ implies that $K\subseteq\opn{Ker}(f\boxtimes id_N)$, so it is enough to show that the natural map $\overline{f\boxtimes id_N}: (M_1\boxtimes N)/K\rightarrow M_2\boxtimes N$ induced by $f\boxtimes id_N$ is injective. Thus we construct a one-sided inverse $M_2\boxtimes N\rightarrow (M_1\boxtimes N)/K$.

 To do so, the universal property in Definition \ref{def:tens-prod} shows we need an intertwining operator of type $\binom{(M_1\boxtimes N)/K}{M_2\,\,N}$. Using notation from Definitions \ref{def:intw-op} and \ref{def:tens-prod}, we define
\begin{align*}
\mcl{Y}: M_2\otimes_{\mathbb{C}} N &\longrightarrow ((M_1\otimes_{\mathbb{C}} N)/K) [\log x]\lbrace x\rbrace\\
m_2\otimes_{\mathbb{C}} n &\longmapsto \sum_{h\in\mathbb{C}}\sum_{k\in\mathbb{Z}_{\geq 0}} ((m_1)^{\mcl{Y}_{M_1,N}}_{h;k} n +K)\,x^{-h-1}(\log x)^k,
\end{align*}
where $f(m_1)=m_2$ (such an $m_1$ exists for any $m_2\in M_2$ because $f$ is surjective). The map $\mcl{Y}$ is well defined because if $f(m_1)=m_2=f(\widetilde{m}_1)$ for $m_1,\widetilde{m}_1\in M_1$, then $m_1-\widetilde{m}_1\in\mathrm{Ker}\,f$ and hence $(m_1-\widetilde{m}_1)^{\mcl{Y}_{M_1,N}}_{h;k} n\in K$ for all $h\in\mathbb{C}$, $k\in\mathbb{Z}_{\geq 0}$, and $n\in N$. It is also easy to see that $\mcl{Y}$ is an intertwining operator since $\mcl{Y}_{M_1,N}$ is an intertwining operator and $f$ is a $V$-module map.

Thus by the universal property of the tensor product $M_2\boxtimes N$, there is a unique $V$-module map $p: M_2\boxtimes N\rightarrow (M_1\boxtimes N)/K$ such that $p\circ\mcl{Y}_{M_2,N}=\mcl{Y}$. Then letting $\pi: M_1\boxtimes N\rightarrow (M_1\boxtimes N)/K$ denote the natural surjection, we have
\begin{align*}
p\circ\overline{f\boxtimes id_N}\circ\pi\circ\mcl{Y}_{M_1,N}(m_1,x)n & = p\circ(f\boxtimes id_N)\circ\mcl{Y}_{M_1,N}(m_1,x)n\\
& = p\circ\mcl{Y}_{M_2,N}(f(m_1),x)n\\
& =\mcl{Y}(f(m_1),x)n\\
& = \pi\circ\mcl{Y}_{M_1,N}(m_1,x)n
\end{align*}
for any $m_1\in M_1$, $n\in N$. Since $M_1\boxtimes N$ is spanned by the coefficients of powers of $x$ and $\log x$ in series of the form $\mcl{Y}_{M_1,N}(m_1,x)n$, and since $\pi$ is surjective, it follows that $p\circ\overline{f\boxtimes id_N}= id_{(M_1\boxtimes N)/K}$, and therefore $\overline{f\boxtimes id_N}$ is injective. This proves that $K=\opn{Ker}(f\boxtimes id_N)$, and $\opn{Ker}(id_N\boxtimes f)$ is determined similarly.
\end{proof}

We now prove Proposition \ref{prop:key-prop}, which as already mentioned will prove Theorem \ref{thm:vertex-mon-for-C0} when combined with arguments from \cite{CMY-completions} and Proposition \ref{prop:VOA-tens-prods-co-cont}:
\begin{proof}[Proof of Proposition \ref{prop:key-prop}]
We prove that $\Phi_{M_A,N_B}$ is an isomorphism, since the proof for $\Psi_{M_A,N_B}$ is similar. We will use Lemma \ref{lem:right-exact} to show that for any $\alpha\in A$, $\beta\in B$, there is a unique $V$-module map
\begin{equation*}
\overline{\phi}_{\alpha,\beta}: (M_\alpha/\opn{Ker} \varphi_\alpha)\boxtimes N_\beta \longrightarrow \varinjlim M_A\boxtimes N_B
\end{equation*}
such that 
\begin{equation*}
\phi_{\alpha,\beta}=\overline{\phi}_{\alpha,\beta}\circ(\pi_\alpha\boxtimes id_{N_\beta}),
\end{equation*}
where $\pi_\alpha: M_\alpha\rightarrow M_\alpha/\mathrm{Ker}\,\varphi_\alpha$ is the natural surjection and $\phi_{\alpha,\beta}: M_\alpha\boxtimes N_\beta\rightarrow\varinjlim M_A\boxtimes N_B$ is the structure map of the direct limit. Since $\pi_\alpha\boxtimes id_{N_\beta}$ is surjective by \cite[Proposition 4.26]{HLZ3}, it is enough to show that $\opn{Ker}(\pi_\alpha\boxtimes id_{N_\beta})\subseteq\opn{Ker} \phi_{\alpha,\beta}$.

By \cite[Lemma 3.5]{CMY-completions} for example, $\mathrm{Ker}\,\varphi_\alpha=\bigcup_{\widetilde{\alpha}\geq\alpha} \opn{Ker} f_\alpha^{\widetilde{\alpha}}$, where as before, $f_\alpha^{\widetilde{\alpha}}: M_\alpha\rightarrow M_{\widetilde{\alpha}}$ is a transition map of $M_A$. Thus by Lemma \ref{lem:right-exact}, $\opn{Ker}(\pi_\alpha\boxtimes id_{N_\beta})$ is spanned by coefficients of powers of $x$ and $\log x$ in series of the form $\mcl{Y}_{M_\alpha,N_\beta}(k_{\widetilde{\alpha}},x)n$ where $k_{\widetilde{\alpha}}\in\opn{Ker} f_\alpha^{\widetilde{\alpha}}$ for some $\widetilde{\alpha}\geq\alpha$ in $A$ and $n\in N$. Then for such $k_{\widetilde{\alpha}}$ and $n$,
\begin{align*}
\phi_{\alpha,\beta}\circ\mcl{Y}_{M_\alpha,N_\beta}(k_{\widetilde{\alpha}},x)n & = \phi_{\widetilde{\alpha},\beta}\circ(f_\alpha^{\widetilde{\alpha}}\boxtimes id_{N_\beta})\circ\mcl{Y}_{M_\alpha,N_\beta}(k_{\widetilde{\alpha}},x)n\\
& =\phi_{\widetilde{\alpha},\beta}\circ\mcl{Y}_{M_{\widetilde{\alpha}},N_\beta}(f_\alpha^{\widetilde{\alpha}}(k_{\widetilde{\alpha}}),x)n = 0,
\end{align*}
showing that $\opn{Ker}(\pi_\alpha\boxtimes id_{N_\beta})\subseteq\opn{Ker}\phi_{\alpha,\beta}$, and thus $\overline{\phi}_{\alpha,\beta}$ is well defined.

Now to construct an inverse of $\Phi_{M_A,N_B}$, take any $\mcl{C}_0$-submodule $W$ of $M=\varinjlim M_A$. By assumption, $W$ is a finitely-generated $V$-module, so because $M=\bigcup_{\alpha\in A} \mathrm{Im}\,\varphi_\alpha$ (see for example \cite[Proposition 3.2]{CMY-completions}), the finitely many generators of $W$ are contained in $\bigcup_{i=1}^N\mathrm{Im}\,\varphi_{\alpha_i}$ for finitely many $\alpha_1,\ldots,\alpha_N\in A$. Then because $A$ is a directed set, there exists $\alpha\in A$ such that $\alpha\geq \alpha_i$ for $i=1,\ldots, N$, and it then follows that $W\subseteq\mathrm{Im}\,\varphi_\alpha$ (see for example \cite[Lemma 4.2]{CMY-completions}). We use $f_W^{\opn{Im}\varphi_\alpha}: W\rightarrow\opn{Im} \varphi_\alpha$ to denote the inclusion, which is a transition map of $M_{\opn{Can}}$. We also use $h_\alpha: M_\alpha/\opn{Ker}\varphi_\alpha\rightarrow\opn{Im}\varphi_\alpha$ to denote the isomorphism induced by $\varphi_\alpha$. Then for $\beta\in B$, we define $\psi_{W,\beta}: W\boxtimes N_\beta\rightarrow\varinjlim M_A\boxtimes N_B$ to be the composition
\begin{equation*}
W\boxtimes N_\beta\xrightarrow{f_W^{\opn{Im} \varphi_\alpha}\boxtimes id_{N_\beta}}\mathrm{Im}\,\varphi_\alpha\boxtimes N_\beta\xrightarrow{h_\alpha^{-1}\boxtimes id_{N_\beta}} (M_\alpha/\mathrm{Ker}\,\varphi_\alpha)\boxtimes N_\beta\xrightarrow{\overline{\phi}_{\alpha,\beta}}\varinjlim M_A\boxtimes N_B.
\end{equation*}
We claim that $\psi_{W,\beta}$ does not depend on the choice of $\alpha\in A$ such that $W\subseteq\mathrm{Im}\,\varphi_\alpha$. 

Indeed, if $W\subseteq\mathrm{Im}\,\varphi_{\alpha_i}$ for $i=1,2$, then there exists $\alpha\in A$ such that $\alpha\geq \alpha_i$ for $i=1,2$. Then
\begin{align*}
\overline{\phi}_{\alpha,\beta} & \circ(h_\alpha^{-1}\boxtimes id_{N_\beta})\circ(f_W^{\opn{Im} \varphi_\alpha}\boxtimes id_{N_\beta}) \nonumber\\
&= \overline{\phi}_{\alpha,\beta}\circ(h_\alpha^{-1}\boxtimes id_{N_\beta})\circ(f_{\opn{Im} \varphi_{\alpha_i}}^{\opn{Im}\varphi_\alpha}\boxtimes id_{N_\beta})\circ(f_{W}^{\opn{Im}\varphi_{\alpha_i}}\boxtimes id_{N_\beta}).
\end{align*}
It is enough to show that the composition of the first three morphisms on the right side is equal to $\overline{\phi}_{\alpha_i,\beta}\circ(h_{\alpha_i}^{-1}\boxtimes id_{N_\beta})$, since this will show that the definition of $\psi_{W,\beta}$ using $\alpha_1$ agrees with that using $\alpha_2$ (and both agree with the definition using $\alpha$).
Since $\varphi_{\alpha_i}\boxtimes id_{N_\beta}$ is surjective onto $\opn{Im} \varphi_{\alpha_i}\boxtimes N_\beta$, it is enough to show that
\begin{align*}
\overline{\phi}_{\alpha,\beta}  \circ(h_\alpha^{-1}\boxtimes id_{N_\beta}) & \circ(f_{\opn{Im} \varphi_{\alpha_i}}^{\opn{Im}\varphi_\alpha}\boxtimes id_{N_\beta})\circ(\varphi_{\alpha_i}\boxtimes id_{N_\beta})\\
& = \overline{\phi}_{\alpha_i,\beta}\circ(h_{\alpha_i}^{-1}\boxtimes id_{N_\beta})\circ(\varphi_{\alpha_i}\boxtimes id_{N_\beta}).
\end{align*}
To prove this, we use the commutative diagram
\begin{equation*}
\xymatrixcolsep{3pc}
\xymatrix{
M_{\alpha_i} \ar[d]^{f_{\alpha_i}^\alpha} \ar[rr]^{\varphi_{\alpha_i}=h_{\alpha_i}\circ\pi_{\alpha_i}} && \opn{Im}\varphi_{\alpha_i} \ar[d]^{f_{\opn{Im}\varphi_{\alpha_i}}^{\opn{Im}\varphi_\alpha}} \\
M_\alpha \ar[r]^(.4){\pi_\alpha} & M_\alpha/\opn{Ker} \varphi_\alpha \ar[r]^(.55){h_\alpha} & \opn{Im} \varphi_\alpha\\
}
\end{equation*}
which implies that $h_\alpha^{-1}\circ f_{\opn{Im}\varphi_{\alpha_i}}^{\opn{Im}\varphi_\alpha}\circ\varphi_{\alpha_i} =\pi_\alpha\circ f_{\alpha_i}^\alpha$. Thus using the definitions,
\begin{align*}
\overline{\phi}_{\alpha,\beta}  \circ(h_\alpha^{-1}\boxtimes id_{N_\beta}) & \circ(f_{\opn{Im} \varphi_{\alpha_i}}^{\opn{Im}\varphi_\alpha}\boxtimes id_{N_\beta})\circ(\varphi_{\alpha_i}\boxtimes id_{N_\beta})\\
& =\overline{\phi}_{\alpha,\beta}\circ(\pi_\alpha\boxtimes id_{N_\beta})\circ(f_{\alpha_i}^\alpha\boxtimes id_{N_\beta})\\
& = \phi_{\alpha,\beta}\circ(f_{\alpha_i}^\alpha\boxtimes id_{N_\beta})\\
& = \phi_{\alpha_i,\beta}\\
& =\overline{\phi}_{\alpha_i,\beta}\circ(\pi_{\alpha_i}\boxtimes id_{N_\beta})\\
& =\overline{\phi}_{\alpha_i,\beta}\circ(h_{\alpha_i}^{-1}\boxtimes id_{N_\beta})\circ(\varphi_{\alpha_i}\boxtimes id_{N_\beta}),
\end{align*}
as required. This proves that the definition of $\psi_{W,\beta}$ does not depend on the choice of $\alpha$ such that $W\subseteq\opn{Im}\varphi_\alpha$.

Now we claim that the maps $\psi_{W,\beta}: W\boxtimes N_\beta\rightarrow\varinjlim M_A\boxtimes N_B$ make $\varinjlim M_A\boxtimes N_B$ a cocone of $M_{\opn{Can}}\boxtimes N_B$. Indeed, for $\mcl{C}_0$-submodules $W_1\subseteq W_2$ of $M$ and $\beta_1\leq \beta_2$ in $B$, we take $\alpha\in A$ such that $W_1\subseteq W_2\subseteq\opn{Im}\varphi_\alpha$ and calculate
\begin{align*}
\psi_{W_2,\beta_2} & \circ(f_{W_1}^{W_2}\boxtimes g_{\beta_1}^{\beta_2})\\
 &=\overline{\phi}_{\alpha,\beta_2}\circ(h_\alpha^{-1}\boxtimes id_{N_{\beta_2}})\circ(f_{W_2}^{\opn{Im}\varphi_\alpha}\boxtimes id_{N_{\beta_2}})\circ(f_{W_1}^{W_2}\boxtimes g_{\beta_1}^{\beta_2})\\
& =\overline{\phi}_{\alpha,\beta_2}\circ(id_{M_\alpha/\opn{Ker}\varphi_\alpha}\boxtimes g_{\beta_1}^{\beta_2})\circ(h_\alpha^{-1}\boxtimes id_{N_{\beta_1}})\circ(f_{W_1}^{\opn{Im}\varphi_\alpha}\boxtimes id_{N_{\beta_1}}).
\end{align*}
We need to show that this equals $\psi_{W_1,\beta_1}$, so it is enough to show that
\begin{equation*}
\overline{\phi}_{\alpha,\beta_2}\circ(id_{M_\alpha/\opn{Ker}\varphi_\alpha}\boxtimes g_{\beta_1}^{\beta_2}) = \overline{\phi}_{\alpha,\beta_1}.
\end{equation*}
This follows from the next calculation since $\pi_\alpha\boxtimes id_{N_{\beta_1}}$ is surjective:
\begin{align*}
\overline{\phi}_{\alpha,\beta_2}\circ(id_{M_\alpha/\opn{Ker}\varphi_\alpha} & \boxtimes g_{\beta_1}^{\beta_2})\circ(\pi_\alpha\boxtimes id_{N_{\beta_1}})\\
& =\phi_{\alpha,\beta_2}\circ(id_{M_\alpha}\boxtimes g_{\beta_1}^{\beta_2})
 =\phi_{\alpha,\beta_1} =\overline{\phi}_{\alpha,\beta_1}\circ(\pi_\alpha\boxtimes id_{N_{\beta_1}}).
\end{align*} 
Thus $\varinjlim M_A\boxtimes N_B$ is a cocone of $M_{\opn{Can}}\boxtimes N_B$, so by the universal property of direct limits, there is a unique map $\widetilde{\Phi}_{M_A,N_B}: \varinjlim M_{\opn{Can}}\boxtimes N_B\rightarrow\varinjlim M_A\boxtimes N_B$ such that
\begin{equation*}
\xymatrixcolsep{5pc}
\xymatrix{
W\boxtimes N_\beta \ar[d]^{\phi_{W,\beta}} \ar[rd]^{\psi_{W,\beta}} & \\
\varinjlim M_{\opn{Can}}\boxtimes N_B \ar[r]^{\widetilde{\Phi}_{M_A,N_B}} & \varinjlim M_A\boxtimes N_B\\
}
\end{equation*}
commutes for $\mcl{C}_0$-submodules $W\subseteq M$ and $\beta\in B$.

We now show that $\widetilde{\Phi}_{M_A,N_B}$ is the inverse of $\Phi_{M_A,N_B}$. To show that $\widetilde{\Phi}_{M_A,N_B}\circ\Phi_{M_A,N_B}$ is the identity on $\varinjlim M_A\boxtimes N_B$, the universal property of direct limits implies that it is enough to show that
\begin{equation*}
\widetilde{\Phi}_{M_A,N_B}\circ\Phi_{M_A,N_B}\circ\phi_{\alpha,\beta}=\phi_{\alpha,\beta}
\end{equation*}
for all $\alpha\in A$, $\beta\in B$. To show this, we observe that since $\opn{Im}\,\varphi_\alpha$ is a $\mcl{C}_0$-submodule of $M$ which is contained in itself, we can use $\alpha$ to define $\psi_{\opn{Im}\varphi_\alpha,\beta}$, and therefore
\begin{align*}
\widetilde{\Phi}_{M_A,N_B}\circ\Phi_{M_A,N_B}\circ\varphi_{\alpha,\beta} & = \widetilde{\Phi}_{M_A,N_B}\circ\phi_{\opn{Im}\varphi_\alpha,\beta}\circ(\varphi_\alpha\boxtimes id_{N_\beta})\\
& =\psi_{\opn{Im}\varphi_\alpha,\beta}\circ(\varphi_\alpha\boxtimes id_{N_\beta})\\
& = \overline{\phi}_{\alpha,\beta}\circ(h_\alpha^{-1}\boxtimes id_{N_\beta})\circ(f_{\opn{Im}\varphi_\alpha}^{\opn{Im}\varphi_\alpha}\boxtimes id_{N_\beta})\circ(\varphi_\alpha\boxtimes id_{N_\beta})\\
& = \overline{\phi}_{\alpha,\beta}\circ(\pi_\alpha\boxtimes id_{N_\beta})\\
& =\phi_{\alpha,\beta},
\end{align*}
as required. Finally, to show that $\Phi_{M_A,N_B}\circ\widetilde{\Phi}_{M_A,N_B}$ is the identity on $\varinjlim M_{\opn{Can}}\boxtimes N_B$, the universal property of direct limits implies that it is enough to show that
\begin{equation*}
\Phi_{M_A,N_B}\circ\widetilde{\Phi}_{M_A,N_B}\circ\phi_{W,\beta}=\phi_{W,\beta}
\end{equation*}
for all $\mcl{C}_0$-submodules $W\subseteq M$ and $\beta\in B$. Thus we choose $\alpha\in A$ such that $W\subseteq\opn{Im}\varphi_\alpha$ and calculate
\begin{align*}
\Phi_{M_A,N_B}\circ\widetilde{\Phi}_{M_A,N_B}\circ\phi_{W,\beta} & = \Phi_{M_A,N_B}\circ\psi_{W,\beta}\\
& = \Phi_{M_A,N_B}\circ\overline{\phi}_{\alpha,\beta}\circ(h_\alpha^{-1}\boxtimes id_{N_\beta})\circ(f_W^{\opn{Im}\varphi_\alpha}\boxtimes id_{N_\beta}).
\end{align*}
We compute the composition of the first three morphisms on the right side by composing them with the surjection $\varphi_\alpha\boxtimes id_{N_\beta}$ onto $\opn{Im}\varphi_\alpha\boxtimes N_\beta$:
\begin{align*}
 \Phi_{M_A,N_B}\circ\overline{\phi}_{\alpha,\beta}\circ(h_\alpha^{-1}\boxtimes id_{N_\beta})\circ(\varphi_\alpha\boxtimes id_{N_\beta}) & = \Phi_{M_A,N_B}\circ\overline{\phi}_{\alpha,\beta}\circ(\pi_\alpha\boxtimes id_{N_\beta})\\
 & =\Phi_{M_A,N_B}\circ\phi_{\alpha,\beta}\\
 & =\phi_{\opn{Im}\varphi_\alpha,\beta}\circ(\varphi_\alpha\boxtimes id_{N_\beta}).
\end{align*}
It follows that
\begin{align*}
\Phi_{M_A,N_B}\circ\widetilde{\Phi}_{M_A,N_B}\circ\phi_{W,\beta} =\phi_{\opn{Im}\varphi_\alpha,\beta}\circ(f_W^{\opn{Im}\varphi_\alpha}\boxtimes id_{N_\beta}) =\phi_{W,\beta},
\end{align*}
as required. This completes the proof that $\widetilde{\Phi}_{M_A,N_B}$ is the inverse of $\Phi_{M_A,N_B}$, so that $\Phi_{M_A,N_B}$ is an isomorphism.
\end{proof}

\bibliographystyle{abbrv}

\end{document}